\documentclass[12pt,a4paper]{article}
\usepackage{graphicx,amsmath,amssymb,epsfig,latexsym,lscape,comment,color,caption,subcaption,titlesec,setspace,changepage,float,amsthm,cite,rotating,bigints,multirow,mathtools,parskip,algorithm,enumitem,algorithm,dsfont,booktabs}

\usepackage{thmtools} 
\usepackage{tcolorbox}
\usepackage[mathcal]{eucal}
\usepackage[colorlinks=true, allcolors=blue]{hyperref}
\usepackage{cleveref}
\usepackage[normalem]{ulem}
\usepackage[dvipsnames]{xcolor}
\usepackage[title]{appendix}
\usepackage{tcolorbox}
\allowdisplaybreaks

\newcommand{\prox}{\ensuremath{\mathrm{Prox}}}

\newtheorem{theorem}{Theorem}[section]

\newtheorem{lemma}[theorem]{Lemma}
\newtheorem{proposition}[theorem]{Proposition}

\newcommand{\st}{\mathrm{s.t.}}
\newcommand{\R}{\mathbb{R}}
\newcommand{\RR}{\overline{\mathbb{R}}}
\newcommand{\norm}[1]{\|#1\|}

\DeclareMathOperator*{\argmin}{arg\,min}
\newcommand{\tr}{\mathrm{tr}}

\usepackage[margin = 0.78in]{geometry}
\newcommand{\vect}[1]{\boldsymbol{#1}}

\usepackage{titlesec}
\titleformat{\section}
  {\normalfont\large\bfseries}  
  {\thesection}{1em}{}
\titleformat{\subsection}
  {\normalfont\normalsize\bfseries}
  {\thesubsection}{1em}{}

\title{\sf\large Interwoven SDP in Primal-Dual Proximal Splitting Methods for Adjustable Robust Convex Optimisation with SOS-Convex Polynomial Constraints }

\author{\small N. D. Dizon\thanks{Department of Applied Mathematics, University of New South Wales, Sydney 2052, Australia. Emails: \url{n.dizon@unsw.edu.au}, \url{b.caldwell@unsw.edu.au},   \url{v.jeyakumar@unsw.edu.au}, \url{g.li@unsw.edu.au}. The research of the authors was supported by a grant from the Australian Research Council.}   \and \small B. I. Caldwell\footnotemark[1] \and \small V. Jeyakumar\footnotemark[1]\thanks{Corresponding author.} \and \small G. Li\footnotemark[1] }
\date{\small First version: November 25, 2025}

\begin{document}
\onehalfspacing
\maketitle
\begin{abstract}
We propose a novel methodology for solving a two-stage adjustable robust convex optimisation problem with a general (proximable) convex objective function and constraints defined by sum-of-squares (SOS) convex polynomials. These problems appear in many decision-making applications. However, they are challenging to solve and typically cannot be reformulated as numerically tractable convex optimisation models, such as conic linear programs, that can be solved directly using existing software. We show that the robust problem admits an equivalent representation as a convex composite unconstrained optimisation model that preserves the same objective values, under quadratic decision rules on the adjustable decision variables. Building on this reformulation, we develop a tailored first-order primal-dual proximal splitting method.  By leveraging semidefinite programming (SDP) techniques as well as tools from convex analysis and real algebraic geometry, we establish its theoretical properties, including computable SDP-based formulas for projections onto closed convex sets, specified by SOS-convex polynomial inequalities. Numerical experiments on a two-stage lot-sizing model with both linear as well as SOS-convex polynomial storage costs under demand uncertainty demonstrate the effectiveness and applicability of the proposed approach. Our approach enables the incorporation of
SDP techniques into a primal-dual proximal splitting framework, thereby broadening the class of problems to which these methods can be effectively applied.
\end{abstract}


\section{Introduction}
Robust optimisation (RO) has emerged as one of the leading modelling frameworks for decision-making under data uncertainty \cite{ben2009robust}. In its classical (static) form, RO seeks solutions that remain feasible for all possible realisations of uncertain parameters within a prescribed uncertainty set. All decision variables in static RO are ``here-and-now” decisions: their values must be determined before the uncertain parameters are revealed \cite{ben2004adjustable}. This approach provides an effective safeguard against uncertainty.

Research in the late 1990s established that RO offers a computationally tractable methodology for many classes of optimisation problems under uncertainty \cite{el1998robust,nemirovski1998robust,ben2009robust}. Since then, the framework has evolved into a widely adopted tool in practice, thanks to its balance between modelling flexibility and numerical tractability. In particular, RO has influenced areas such as machine learning \cite{dunbar2010simultaneous,jeyakumar2014support}, portfolio management \cite{goldfarb2003robust}, and supply chain and logistics \cite{ben2009robust}. We refer the reader to \cite{bertsimas2011theory} for a detailed overview of RO theory and applications.

Recent advances have extended static RO models to practical applications in healthcare and medical decision-making, such as radiation therapy planning \cite{fredriksson2017robust}, disease characterisations \cite{dunbar2010simultaneous, cysique2010screening, woolnough2022robust}. These models are designed to guard against all possible realisations of patient-treatment-related uncertainties. Notably, a static RO framework with evolving, time-dependent uncertainty sets was proposed in \cite{nohadani2017robust} and successfully demonstrated its effectiveness in prostate cancer treatment planning. A broader discussion of robust treatment planning is given in \cite{fredriksson2017robust}.

However, static RO models can lead to overly conservative decisions \cite{roos2020reducing}. For instance, robust radiation therapy plans may require excessively high doses, making treatments more costly and less practical. Adjustable robust optimisation (ARO) was introduced as an extension of RO for multi-stage decision-making problems \cite{delage2015robust,ben2004adjustable} that can reduce conservatism. In ARO, some decisions remain``here-and-now”, while others are modelled as ``wait-and-see” decisions, determined after partial uncertainty is resolved. ARO guarantees that its worst-case objective value is no worse than that of static RO \cite{marandi2018static,ben2009robust}. 

Yet, its main challenge lies in the optimisation of decision rules that arise from allowing the ``wait-and-see" decisions to adapt to uncertainty. These rules are mappings rather than fixed vectors. Optimising over mappings is theoretically and computationally demanding. To address this, tractable subclasses such as affine or quadratic decision rules have been proposed, enabling numerically tractable reformulations for many classes of problems \cite{chuong2021exact,jeyakumar2024affinely,jeyakumar2021quadratically,woolnough2021exact}. The ARO framework has recently been applied to radiotherapy planning in  \cite{jeyakumar2024affinely,ten2022adjustable}. A survey provides an overview of ARO and its applications \cite{yanikouglu2019survey}.
 
The primary computational approach in RO and ARO is to reformulate a given uncertain problem into its robust counterpart, considering the problem's underlying structures of both objective and constraint functions. Under appropriate conditions, robust counterparts often result in conic convex programs \cite{ben2009robust,jeyakumar2015robust,jeyakumar2021quadratically,jeyakumar2024affinely} that are solvable using standard optimisation software. Nevertheless, this approach limits the problem classes to those that admit such reformulations, and, in practice, the reformulations suffer from scalability compared to their nominal counterparts. For example, the robust counterpart of a linear program with ellipsoidal uncertainty becomes a convex quadratic program \cite{ben2009robust}, which is less scalable in large-scale contexts, such as in machine learning.

To address scalability, recent research has turned to first-order methods \cite{beck2017first} such as projection-type methods \cite{bauschke2020dykstra,aragon2018new,svaiter2011weak}, which exploit subgradient or proximal operator information to efficiently solve large-scale problems. These methods avoid the computational burden of solving a full optimisation problem at every iteration, relying instead on inexpensive first-order oracles. Such efficiency is particularly valuable in areas like machine learning \cite{woolnough2022robust} and supply chains \cite{woolnough2021exact,jeyakumar2021quadratically}, where RO is increasingly applied to large and complex models. 

Motivated by these developments, this paper presents a primal-dual proximal splitting (PDPS) method for solving the following ARO problem:
\begin{align}\label{problem:ro-adjustable-intro}\tag{ARP}
    \min_{\vect{x} \in \mathbb{R}^d,\, \vect{y}(\cdot)}
        & 
        \ f(\vect{x}) \\
        \text{s.t.} \
        & \vect{a}_i(\vect{w})^\top \vect{x} + \vect{c}_i^\top \vect{y}(\vect{w}) \le b_i(\vect{w}), && \forall \vect{w} \in \mathcal{B}, \ i = 1,\ldots,m, \notag
        \\
        &
        \vect{x}\in \mathcal{C}, \; \;g_j(\vect{x}) \le 0, && j=1,\ldots,s, \notag
\end{align} 
where $f:\R^d \to \ ]\!-\!\infty, +\infty]$ is a proper, convex, lower semicontinuous, proximable function in the sense that its proximal operator is well-defined and efficiently computable, the set $\mathcal{C}\subset\mathbb{R}^d$ is a closed convex set for which the projection onto $\mathcal{C}$ is efficiently computable. We refer the reader to \cite{BauCom2017,beck2017first} for examples of efficiently computable proximal and projection operators. The functions $g_j:\mathbb{R}^d\to \R, \; j=1, \ldots, s$, are \textit{SOS-convex polynomials}.

The robust inequality is an affinely parameterised system, where  $\vect{a}_i(\vect{w}) = \vect{a}_i^{(0)} + \sum_{\ell=1}^k w^{(\ell)}\vect{a}_i^{(\ell)}$, $i=1,\ldots,m$, $\vect{a}_i^{(\ell)} \in \R^d$, $\ell=0,\ldots,k$, $i=1,\ldots,m$, $b_i(\vect{w}) = b_i^{(0)} + \sum_{\ell=1}^k w^{(\ell)} b_i^{(\ell)}$, $i=1,\ldots,m$, and $b_i^{(\ell)} \in \R$, $\ell=0,\ldots,k$, $i=1,\ldots,m$. Note that, in \eqref{problem:ro-adjustable-intro}, $\vect{x}$ represents the first-stage ``here-and-now'' decision that is made before $\vect{w} \in \mathcal{B}$ is realised, $\vect{y}$ represents the ``wait-and-see'' decision that can be adjusted according to the actual data, $\vect{c}_i \in \R^q$, $i=1,\ldots,m$, and the uncertainty set $\mathcal{B}:= \{\vect{w}\in \R^k : \|\vect{w} -\vect{d}\|^2_2 \le r\}$ is a Euclidean ball, for some fixed $r>0$ and $\vect{d} \in \R^k$. The problem \eqref{problem:ro-adjustable-intro} requires optimising a convex function over mappings $\vect{y}: \mathcal{B}\to \R^q$ rather than vectors. In general, obtaining a numerically tractable reformulation of such systems is difficult unless the mapping $\vect{y}(\cdot)$ is restricted to specific classes of mappings via decision rules \cite{woolnough2021exact,jeyakumar2021quadratically,ben2009robust,chuong2021exact}. 

The \eqref{problem:ro-adjustable-intro} has constraints involving SOS-convex polynomials, which are a relatively new class of convex polynomials \cite{ahmadi2012convex,jeyakumar2015robust}, encompassing the class of affine functions, convex quadratic functions, and convex separable polynomials. In practice, constraints involving SOS-convex polynomials have appeared as convex quadratic functions in portfolio selection models, separable convex polynomials in lot-sizing problems \cite{jeyakumar2021quadratically,woolnough2021exact} (see also Section 7) or affine functions in inventory-production management problems \cite{chuong2021exact}.

The SOS-convexity is a numerically tractable relaxation of convexity for polynomials, and it allows efficient certification of nonnegativity through semidefinite programs (SDPs). Thus, SOS-convexity provides a novel framework for representing polynomial convexity, analogous to the fundamental role of the sum-of-squares property in polynomial nonnegativity. A recent extension of SOS-convexity and its applications to distributionally robust optimisation can be found in \cite{huang2024piecewise}.

\medskip

\textbf{Contributions}. The main contributions are itemised below:
\begin{enumerate}[label= \it(\roman*)]
    \item By reformulating the robust constraints into linear matrix inequalities (LMIs) with the aid of a generalised $S$-lemma, we transform the robust counterpart of \eqref{problem:ro-adjustable-intro} (under a quadratic decision rule) into a convex composite unconstrained optimisation problem, sharing the same optimal values as the original robust problems. This reformulation allows the application of a PDPS method for solving \eqref{problem:ro-adjustable-intro}.

    \item By leveraging semidefinite optimisation techniques as well as tools from convex analysis and real algebraic geometry, we present new computable SDP-based formulas for projections onto closed convex sets, specified by SOS-convex polynomial inequalities, including convex quadratic inequalities as a special case. These formulas enable the integration of SDPs within the first-order proximal splitting algorithms for solving robust convex optimisation problems.

    \item We demonstrate the practicality and computational efficiency of the proposed framework by performing numerical experiments on a lot-sizing problem with linear as well as SOS-convex storage costs under demand uncertainty, formulated as a two-stage ARO model. The results reveal that the PDPS approach with interwoven SDP-based calculations effectively handles nonlinear SOS-convex costs, maintaining tractability for moderate-scale robust lot-sizing problems. 

\end{enumerate}
\textbf{Novelty}. The originality of this work lies in deriving SDP-based formulas for computing the projection onto closed convex sets, described by the relatively recent class of SOS-convex polynomials, thereby enabling the integration of semidefinite programming techniques within a primal-dual splitting framework for robust convex optimisation. To the best of our knowledge, no previous studies have established an SDP-based formula for the computation of projections onto closed convex sets described by SOS-convex inequalities. Building on this result, we implement a first-order PDPS method for a broad class of robust convex optimisation problems, which, in general, cannot be equivalently reformulated as numerically tractable convex programs, such as SDPs, efficiently solvable by existing software.

A key feature of the proposed framework is the ability to compute proximal points through SDPs, allowing the evaluation of a broad class of projection operators. The combined use of SDPs and SOS-convexity establishes a flexible modelling framework that enhances both the numerical tractability of robust convex optimisation and the applicability of PDPS schemes.

\textbf{Organisation}. The organisation of the paper is as follows. To make the paper self-contained, Section~\ref{sec:prelims} provides key convex analysis and real algebraic geometry tools that are used throughout the paper. 
Section~\ref{sec:aro} provides corresponding reformulations for ARO models. Section~\ref{sec:projections-via-SDP} gives formulas for projection and proximal operators. Section~\ref{sec:PD} describes the PDPS method, interwoven with SDPs. Section~\ref{sec:numerics} describes how our framework is used to solve a class of lot-sizing problems in the face of demand uncertainty. Section~\ref{sec:conclusion} concludes with a discussion on future research. Appendix presents technical details related to polynomial systems.

\section{Key Convex Analysis and Algebraic Geometry  Tools}\label{sec:prelims}
In this section, to make the paper self-contained, we provide key convex analysis and basic real algebraic geometry tools that are used throughout the paper. We start by presenting notation, definitions, and fundamental results that will be used later in the paper.

Denote $\R^d$ the Euclidean space of dimension $d$, $\R^d_+$ the nonnegative orthant of $\R^d$, $\vect{a}^\top \vect{b}$ for $\vect{a}$, $\vect{b} \in \R^d$, the standard inner product on $\R^d$, and $\|\vect{a}\|_2$ the Euclidean norm of $\vect{a} \in \mathbb{R}^d$. Denote by $\RR := \,]\!-\!\infty, +\infty]$ the extended real line. Let $\vect{e}_j^{[d]}$ be the $j$\textsuperscript{th} standard basis vector in $\R^d$, and $\vect{0}$ the vector of zeros in appropriate dimensions. Denote $\mathbb{S}^d$ the space of $(d\times d)$ real symmetric matrices, and $\tr(AB)$, $A,B \in \mathbb{S}^d$, the trace product on $\mathbb S^d$. A matrix $B \in \mathbb S^d$ is positive semidefinite, denoted as $B \succeq 0$ (resp. positive definite, denoted as $B \succ 0$), if $\vect{x}^\top B \vect{x} \geq 0$  for all $\vect{x}\in \R^d$ (resp. $\vect{x}^\top B \vect{x} >0$ for all $\vect{x} \in \R^d$, $\vect{x} \neq \vect{0}$). Let $\mathbb{S}_+^d$ be the cone of symmetric $(d \times d)$ positive semidefinite matrices, and $\mathbb{S}_{++}^d$ the open cone of symmetric $(d \times d)$ positive definite matrices. The $(d\times d)$ identity matrix is denoted by $I_{d}$.  The norm induced by the trace inner product is the Frobenius norm, which we denote by $\|\cdot\|_F$.

\textbf{Proximal operators}. Let $V$ be a finite-dimensional Hilbert space. For any $\tau > 0$, the \emph{proximal operator} of $f$ is defined by
\[
    \prox_{\tau f}(\vect{v}) := \argmin_{\vect{u} \in V} \left\{ f(\vect{u}) + \frac{1}{2\tau} \|\vect{u} - \vect{v}\|_V^2 \right\}.
\]
It is known that $\prox_f$ is firmly non-expansive whenever $f$ is proper, convex, and lower semicontinuous \cite[Proposition~12.28]{BauCom2017}. Moreover, if $f = \iota_{\mathcal{C}}$ for some nonempty, closed, convex set $\mathcal{C} \subseteq V$, then $\prox_{f} = P_{\mathcal{C}}$ where $P_{\mathcal{C}}$ is the \textit{projection operator} onto $\mathcal{C}$. Here, the indicator function $\iota_{\mathcal{C}}$ is defined as $\iota_{\mathcal{C}}(\vect{v})=0$ if $\vect{v} \in {\mathcal{C}}$ and $\iota_{\mathcal{C}}(\vect{v})=+\infty$ otherwise.

\textbf{SOS polynomials}. Let $\R [\vect{x}]$ be the space of polynomials with real coefficients over $\vect{x} \in \R^d$. A polynomial $f \in \R[\vect{x}]$ is called a Sum-of-Squares (SOS) polynomial if there exist polynomials $f_j \in \R[\vect{x}]$, $j = 1,\ldots,s$, such that $f = \sum_{j=1}^s f_j^2$, for some $s \in \mathbb{N}$. For a polynomial $f \in \R[\vect{x}]$, we use ${\rm deg}\, f$ to denote its degree. We also use $\Sigma_m^2 (\vect{x})$ to denote the set of all SOS polynomials $f$ of degree at most $m$ with respect to the variable $\vect{x}\in \R^d$. Technical results related to SOS polynomials and LMIs are given in Appendix~\ref{sec:sos}. Next, we recall the definition of SOS-convex polynomials.

\textbf{SOS-convex polynomials} \cite{ahmadi2013complete,helton2010semidefinite}.  A polynomial $f \in \R[\vect{x}]$ is SOS-convex if its Hessian $H(\vect{x})$ is an SOS matrix polynomial, that is, if there exists a $(\nu \times d)$ polynomial matrix $P(\vect{x})$ for some $\nu \in \mathbb N$ such that $H(\vect{x}) = P (\vect{x})^\top P(\vect{x})$.

Several equivalent conditions for SOS-convexity can be found in \cite{ahmadi2013complete}. For instance, 
$f \in \R[\vect{x}]$ is SOS-convex whenever the polynomial $g(\vect{v}, \vect{x}) = f(\vect{v}) - f(\vect{x}) - \nabla f(\vect{x})^\top (\vect{v} - \vect{x})$ on $\R^d \times \R^d$ is an SOS polynomial with respect to the variable $(\vect{v}, \vect{x})$. 

An SOS-convex polynomial is a convex polynomial, but the converse is not true \cite{ahmadi2013complete}. In other words, the class of SOS-convex polynomials is a proper subclass of convex polynomials. While the class of SOS-convex polynomials covers affine functions, convex quadratic functions, and convex separable polynomials, they may also be non-quadratic and non-separable \cite{jeyakumar2015robust}. 

The following proposition provides the key property of SOS-convex polynomials that will be used later in the paper.

\begin{proposition}[{\bf SOS and nonnegative SOS-convex polynomials}] (See \cite{helton2010semidefinite} and \cite[Corollary~2.1]{jeyakumar2014dual}). \label{prop:nn_sos_conv}
    Let $f \in \R[\vect{x}]$ be a nonnegative SOS-convex polynomial. Then, $f$ is an SOS polynomial. 
\end{proposition}

\textbf{Convex semi-algebraic sets}. A (basic) convex semi-algebraic set is a convex set defined by a finite number of polynomial inequalities, and has the form   $\mathcal{D}=\{\vect{x}\in \R^d \; : \; g_j(\vect{x}) \le 0, \; j=1,\ldots, s\}$, where $\mathcal{D}$ is a convex set and $g_j\in\R[\vect{x}]$,  $j=1,\ldots, s$, are real polynomials. Convex semi-algebraic sets play a key role in polynomial optimisation and real algebraic geometry. The common convex semi-algebraic sets are spectrahedra \cite[Section~6.2]{blekherman2012semidefinite} and polyhedra, which frequently appear in robust optimisation as uncertainty sets. A reader is referred to  \cite{blekherman2012semidefinite,lasserre2015introduction} for further details.

\begin{proposition}[{\bf Inhomogeneous $S$-Lemma}](See \cite[Proposition~4.10.1]{ben2001lectures})\label{lem:s-lemma}
Let $A,B\in \mathbb{S}^d, \vect{a},\vect{b}\in \mathbb{R}^d$, $\alpha, \beta \in \mathbb{R}$, and let $f(\vect{x}) = \vect{x}^\top A \vect{x} + 2\vect{a}^\top \vect{x} + \alpha$ and $g(\vect{x}) = \vect{x}^\top B \vect{x} + 2\vect{b}^\top \vect{x} + \beta$ be two quadratic functions. Suppose that
$f(\vect{x}_0) < 0$, for some $\vect{x}_0\in \R^d$. Then, 
\[[\; f(\vect{x}) \le 0 \implies g(\vect{x})\le 0 \;] \iff \exists \lambda\in \R_+, \; \; 
    \lambda\begin{bmatrix}
        A & \vect{a} \\ \vect{a}^\top  & \alpha
    \end{bmatrix} -  \begin{bmatrix}
        B & \vect{b} \\  \vect{b}^\top & \beta
    \end{bmatrix}\succeq 0.
\]
\end{proposition}

    \section{Adjustable Robust Optimisation  Problems}\label{sec:aro}

Consider the fixed-recourse two-stage convex adjustable robust optimisation problem of the form:
\begin{align}\label{problem:ro-adjustable}\tag{ARP}
    \min_{\vect{x} \in \mathbb{R}^d,\, \vect{y}(\cdot)}
        & 
        \ f(\vect{x}) \\
        \text{s.t.} \
        & (\vect{a}_i^{(0)})^\top \vect{x} + \sum_{\ell=1}^k w^{(\ell)}(\vect{a}_i^{(\ell)})^\top \vect{x} + \vect{c}_i^\top \vect{y}(\vect{w})  \le b_i^{(0)} + \sum_{\ell=1}^{k} w^{(\ell)} b_i^{(\ell)}
        && \forall \vect{w} \in \mathcal{B}, \ i = 1,\ldots,m, \notag
        \\
        &
        \vect{x}\in \mathcal{C}, \; \; g_j(\vect{x}) \le 0, && j=1,\ldots,s, \notag
\end{align} 
where $f:\R^d \to \overline{\R}$ is a proximal proper, convex, lower semicontinuous function, $g_j, j=1, \ldots, s$, are SOS-convex polynomials, $\vect{a}_i^{(\ell)} \in \R^d$, $\ell=0,\ldots,k$, $i=1,\ldots,m$, $b_i^{(\ell)} \in \R$, $\ell=0,\ldots,k$, $i=1,\ldots,m$, $\vect{c}_i \in \R^q$, $i=1,\ldots,m$, and $\mathcal{B}:= \{\vect{w}\in \R^k : \|\vect{w} -\vect{d}\|^2_2 \le r\}$.  We also assume that the projection onto the closed convex set $\mathcal{C}$ admits a closed-form expression. In this model, $\vect{x}$ represents the first-stage ``here-and-now'' decision that is made before $\vect{w} \in \mathcal{B}$ is realised and $\vect{y}$ represents the ``wait-and-see'' decision that can be adjusted according to the actual data.

In general, obtaining a numerically tractable reformulation of such systems is difficult unless $\vect{y}(\cdot)$ is restricted to specific classes of mappings. A common approach is to impose affine decision rules of the form $\vect{y}(\vect{w})=\vect{y_0}+U\vect{w}$ with $\vect{y}_0 \in \R^q$ and $U \in \R^{q \times k}$ to be optimised.  For many problems, especially those with affine parameterisations, the linear decision rules lead to computationally tractable reformulations \cite{ben2009robust}. The two-stage linear optimisation problems with quadratic decision rules have also been shown to admit SDP reformulations sharing the same optimal values \cite{woolnough2021exact}. Building on this framework, we develop in this section a numerically tractable LMI-based reformulation by employing a quadratic decision rule for ARP with constraint-wise uncertainty.

By applying the quadratic decision rule (QDR), 
\[
\vect{y}(\vect{w})= \begin{bmatrix}
    \rho y_0^{{(1)}}+ \rho(U\vect{w})^{{(1)}} + (1-\rho)\vect{w}^\top \Theta_1 \vect{w} \\
    \vdots \\
    \rho y_0^{{(q)}}+ \rho (U\vect{w})^{{(q)}} + (1-\rho)\vect{w}^\top \Theta_q \vect{w}
\end{bmatrix},
\]
for some fixed $\rho \in [0,1]$, where $\vect{y}_0 \in \R^q$, $U \in \R^{q \times k}$, and $\Theta_p \in \mathbb{S}^k$, $p=1,\ldots,q$, the problem \eqref{problem:ro-adjustable} turns into the following robust convex optimisation problem:
\begin{align}\label{problem:ro-adjustable-qdr}\tag{P-QDR}
    \min_{\widehat{\vect{x}}}
        & 
        \ f(\vect{x}) + \iota_{\mathcal{C}}(\vect{x}) + \iota_{\mathcal{D}}(\vect{x})\\
        \text{s.t.} \
        & \vect{w}^\top P_i(\widehat{\vect{x}})\vect{w} + 2\vect{q}_i(\widehat{\vect{x}})^\top \vect{w} + \beta_i(\widehat{\vect{x}}) \ge 0,
        && \forall \vect{w} \in \mathcal{B}, \ i=1,\ldots, m, \notag
\end{align}
where $\widehat{\vect{x}} := (\vect{x}, \vect{y}_0, U, \vect{\Theta}) \in \R^d \times \R^q \times \R^{q \times k} \times \prod_{p=1}^q \mathbb{S}^{k}$ with $\vect{\Theta} = (\Theta_1, \ldots, \Theta_q)$, $P_i(\widehat{\vect{x}}) := -(1-\rho)\sum_{p=1}^q c_i^{(p)} \Theta_p$, $\vect{q}_i(\widehat{\vect{x}}):= -\tfrac{1}{2} \Big(\rho U^\top \vect{c}_i + \begin{bmatrix}(\vect{a}_i^{(1)})^\top \vect{x} - b_i^{(1)} & \cdots & (\vect{a}_i^{(k)})^\top \vect{x} - b_i^{(k)} \end{bmatrix}^\top \Big)$ and $\beta_i(\widehat{\vect{x}}) := -(\vect{a}_i^{(0)})^\top \vect{x} + b_i^{(0)} - \rho \vect{c}_i^\top \vect{y}_0$, $i=1,\ldots,m$.

We now present the unconstrained convex composite reformulation for ARO problems with QDR and a ball uncertainty set. 

We begin by deriving, in Lemma \ref{thm:CR-ball-quad} below, an LMI characterisation of the robust affine systems appearing in \eqref{problem:ro-adjustable-qdr} under ball uncertainty, which is a key step in our reformulation. This LMI characterisation is of independent interest, as similar systems arise in several other contexts \cite{zhen2022robust}. For example, the systems  \eqref{problem:ro-adjustable-qdr} appear also in finite-dimensional reduction problems of distributionally robust optimisation problems where $\mathcal{B}$ serves as a support set of distributions \cite{huang2024piecewise,mohajerin2018data}. Moreover, systems of this type also feature in robust set containment characterisations \cite{jeyakumar2010characterizing,jeyakumar2003characterizing}, which, in turn, underpin knowledge-based data classification \cite{mangasarian1997mathematical}.

\begin{lemma}[{\bf LMI Characterisation of Quadratic-Ball Robust System}]\label{thm:CR-ball-quad}
Let $\widehat{\vect{x}} \in \R^\eta$, $P_i: \R^\eta \to \mathbb{S}^{k}$, $\vect{q}_i: \R^\eta \to \R^k$ and $\beta_i: \R^\eta \to \R$ be affine functions of $\widehat{\vect{x}} \in \R^\eta$ for $i=1,\ldots,m$, and let $r >0$. For each fixed $\widehat{\vect{x}} \in \R^\eta$, the following statements are equivalent:
\begin{enumerate}[label=(\roman*)]
\item $ \|\vect{w} -\vect{d}\|^2_2 \le r \implies \displaystyle\min_{i=1,\ldots, m}\{\vect{w}^\top P_i(\widehat{\vect{x}})\vect{w} + 2\vect{q}_i(\widehat{\vect{x}})^\top \vect{w} + \beta_i(\widehat{\vect{x}})\} \ge 0$.
\item There exist $\lambda_i \ge 0$, $i=1,\ldots,m$, such that
\[
\begin{bmatrix}
\lambda_i I_k + P_i(\widehat{\vect{x}}) & -\lambda_i \vect{d} + \vect{q}_i(\widehat{\vect{x}}) \\
(-\lambda_i \vect{d} + \vect{q}_i(\widehat{\vect{x}}))^\top & \lambda_i(\|\vect{d}\|_2^2 - r) + \beta_i(\widehat{\vect{x}})
\end{bmatrix} \succeq 0, \quad i=1,\ldots, m.
\]
\end{enumerate}
\end{lemma}

\begin{proof} Fix $\widehat{\vect{x}} \in \R^\eta$. Firstly, notice that statement \textit{(i)} is equivalent to 
\begin{equation}\label{eqn:qball-i-equiv}
\|\vect{w} -\vect{d}\|^2_2 \le r \implies \vect{w}^\top P_i(\widehat{\vect{x}})\vect{w} + 2\vect{q}_i(\widehat{\vect{x}})^\top \vect{w} + \beta_i(\widehat{\vect{x}}) \ge 0
\end{equation}
for each $i=1,\ldots,m$. Now, fix $i=1,\ldots, m$. Define the quadratic functions:
\[
\widetilde{f}(\vect{w}) := \vect{w}^\top I_k \vect{w} - 2\vect{d}^\top \vect{w} + \|\vect{d}\|_2^2 - r
\quad \text{and} \quad
\widetilde{g}_i(\vect{w}) := -\vect{w}^\top P_i(\widehat{\vect{x}}) \vect{w}
-2\vect{q}_i(\widehat{\vect{x}})^\top \vect{w}
-\beta_i(\widehat{\vect{x}}).
\]
Then, \Cref{eqn:qball-i-equiv} is equivalent to the implication
$
\widetilde{f}(\vect{w}) \le 0 \implies \widetilde{g}_i(\vect{w}) \le 0.
$ 
Since $r>0$, \Cref{lem:s-lemma} (Inhomogeneous $S$-Lemma) applies and the implication is equivalent to the existence of $\lambda_i \ge 0$ such that 
\begin{equation}\label{eqn:qball-ii}
\lambda_i 
\begin{bmatrix} I_k & -\vect{d} \\ -\vect{d}^\top &  \|\vect{d}\|_2^2 - r \end{bmatrix} + 
\begin{bmatrix} P_i(\widehat{\vect{x}}) & \vect{q}_i(\widehat{\vect{x}}) \\ \vect{q}_i(\widehat{\vect{x}})^\top & \beta_i(\widehat{\vect{x}}) \end{bmatrix}
\succeq 0,
\end{equation}
which is precisely the LMI in statement \textit{(ii)}. Since the equivalence between
\eqref{eqn:qball-i-equiv} and \eqref{eqn:qball-ii} holds for arbitrary $i=1,\ldots,m$ and for any fixed $\widehat{\vect{x}}$, then the equivalence of statements \textit{(i)} and \textit{(ii)} follows.
\end{proof}

We now show that problem \eqref{problem:ro-adjustable-qdr} share the same optimal values with its convex composite reformulation given by:
\begin{equation*}\label{eqn:comp_prob_adjustable_qdr}\tag{CCP}
\min_{\widetilde{\vect{x}}\in\widetilde{X}}
\Big\{
F(\widetilde{\vect{x}}) + E(\widetilde{\vect{x}}) + H(\widetilde{\vect{x}}) + G\big(K\widetilde{\vect{x}}\big)
\Big\},
\end{equation*}
where $\widetilde{X}
:= \R^d \times \R^q \times \R^{q\times k} \times \prod_{p=1}^q\mathbb{S}^{k} \times \R^{m}$. For $\widetilde{\vect{x}}
= \big(\vect{x},\,\vect{y}_0,\,U, \vect{\Theta}, \boldsymbol{\lambda}\big) \in \widetilde{X}$ with $\vect{\Theta}=(\Theta_1,\ldots,\Theta_q)$, $\vect{\lambda}=(\lambda_1,\ldots,\lambda_m)$ and $\vect{\widetilde{y}} = (\Psi_1, \ldots, \Psi_m) \in \widetilde{Y} := \prod_{i=1}^m\mathbb{S}^{k+1}$, the proper, convex, lower semicontinuous functions $F,E,H:\widetilde{X}\to\overline{\R}$ and $G:\widetilde{Y}\to\overline{\R}$ are given by 
\begin{equation}\label{eqn:FEHG-SDP2}
F(\widetilde{\vect{x}}):= f(\vect{x}) + \sum_{i=1}^m \iota_{\R_+}(\lambda_i), \ E(\widetilde{\vect{x}})
:= \iota_{\mathcal{C}}(\vect{x}), \ H(\widetilde{\vect{x}})
:= \iota_{\mathcal{D}}(\vect{x}), \ 
G(\widetilde{\vect{y}}):= \sum_{i=1}^m\iota_{B_i + \mathbb{S}_+^{k+1}}(\Psi_i)
\end{equation}
where 
\begin{equation*}
B_i = \begin{bmatrix} 0 & -\tfrac{1}{2}\boldsymbol{b}_i\\ -\tfrac{1}{2}\boldsymbol{b}_i^\top & -b_i^{(0)}\end{bmatrix}, \ \vect{b}_i = \begin{bmatrix} b_i^{(1)} & \cdots & b_i^{(k)} \end{bmatrix}^\top, \ i=1,\ldots, m.
\end{equation*}

Furthermore, the linear mapping $K:\widetilde{X}\to\widetilde{Y}$ is given by 
\begin{equation}\label{eqn:K-SDP2}
K\widetilde{\vect{x}}
:=\big( \Psi_1(\widetilde{\vect{x}}),\ldots,\Psi_m(\widetilde{\vect{x}})\big),
\end{equation}where, for each $i=1,\ldots,m$,
\[
\Psi_i(\widetilde{\vect{x}})
=
\begin{bmatrix}
\lambda_i I_k - (1-\rho)\displaystyle\sum_{p=1}^q c_i^{(p)} \Theta_p
&
-\lambda_i \vect{d} - \tfrac{1}{2}\!\left(\rho U^\top \vect{c}_i + A_i^\top \vect{x}\right)
\\[4pt]
\big(-\lambda_i \vect{d} - \tfrac{1}{2}(\rho U^\top \vect{c}_i + A_i^\top \vect{x})\big)^\top
&
\lambda_i(\|\vect{d}\|_2^2 - r) - (\vect{a}_i^{(0)})^\top \vect{x} - \rho \vect{c}_i^\top \vect{y}_0
\end{bmatrix}
\in \mathbb{S}^{k+1},
\]
with $A_i:= \begin{bmatrix} \vect{a}_i^{(1)} & \cdots &  \vect{a}_i^{(k)} \end{bmatrix}$, $i=1,\ldots,m$. 

\begin{theorem}[\bf Convex Composite form of \eqref{problem:ro-adjustable-qdr}]\label{prop:ro-adjustable-sdp}
Consider the problem \eqref{problem:ro-adjustable-qdr}. Let $\rho \in [0,1]$ and $\mathcal{B} = \{\vect{w} \in \R^k : \|\vect{w} - \vect{d}\|_2^2 \le r\}$, for some $r>0$. Let $F$, $E$, $H$, $G$, and $K$ be as defined in \Cref{eqn:FEHG-SDP2,eqn:K-SDP2}. Then
\begin{equation*}
\min \eqref{problem:ro-adjustable-qdr}
=
\min_{\widetilde{\vect{x}}\in\widetilde{X}}
\Big\{
F(\widetilde{\vect{x}}) + E(\widetilde{\vect{x}}) + H(\widetilde{\vect{x}}) + G\big(K\widetilde{\vect{x}}\big)
\Big\}.
\end{equation*}
\end{theorem}

\begin{proof} Fix an arbitrary $\widehat{\vect{x}} = (\vect{x}, \vect{y}_0, U, \vect{\Theta}) \in \R^d \times \R^q \times \R^{q \times k} \times \prod_{p=1}^q \mathbb{S}^{k}$ where $\vect{\Theta} = (\Theta_1, \ldots, \Theta_q)$. For each $i=1,\ldots,m$, the robust constraints of \eqref{problem:ro-adjustable-qdr} is equivalent to the implication
\begin{equation}\label{eqn:two-stage-cons}
    \|\vect{w} -\vect{d}\|^2_2 \le r \implies \vect{w}^\top P_i(\widehat{\vect{x}})\vect{w} + 2\vect{q}_i(\widehat{\vect{x}})^\top \vect{w} + \beta_i(\widehat{\vect{x}}) \ge 0, \ i=1,\dots, m,
\end{equation}
where $P_i(\widehat{\vect{x}}) := -(1-\rho)\sum_{p=1}^q c_i^{(p)} \Theta_p$, $\vect{q}_i(\widehat{\vect{x}}):= -\tfrac{1}{2} \big(\rho U^\top \vect{c}_i + \big[(\vect{a}_i^{(1)})^\top \vect{x} - b_i^{(1)} \, \cdots \, (\vect{a}_i^{(k)})^\top \vect{x} - b_i^{(k)} \big]^\top \big)$, and $\beta_i(\widehat{\vect{x}}) := -(\vect{a}_i^{(0)})^\top \vect{x} + b_i^{(0)} - \rho \vect{c}_i^\top \vect{y}_0$.
Since $r>0$, \Cref{thm:CR-ball-quad} applies and the implication is equivalent to the existence of $\lambda_i \ge 0$, $i=1,\ldots,m$ such that 
\begin{equation*}
\begin{bmatrix}
\lambda_i I_k + P_i(\widehat{\vect{x}}) & -\lambda_i \vect{d} + \vect{q}_i(\widehat{\vect{x}}) \\
(-\lambda_i \vect{d} + \vect{q}_i(\widehat{\vect{x}})^\top & \lambda_i(\|\vect{d}\|_2^2 - r) + \beta_i(\widehat{\vect{x}})
\end{bmatrix} \succeq 0, \quad i=1,\ldots, m.
\end{equation*}
Substituting for $P_i(\widehat{\vect{x}})$, $\vect{q}_i(\widehat{\vect{x}})$ and $\beta_i(\widehat{\vect{x}})$ in \eqref{eqn:two-stage-cons-equiv} yields the LMI constraints
\begin{align}\label{eqn:two-stage-cons-equiv}
        \begin{bmatrix}
        \lambda_i I_k  - (1-\rho)\sum_{p=1}^q c_i^{(p)} \Theta_p & -\lambda_i \vect{d}  -\tfrac{1}{2} \left(\rho U^\top \vect{c}_i + A_i^\top \vect{x} - \vect{b}_i\right) \\[6pt]
        \big(-\lambda_i \vect{d}  -\tfrac{1}{2} \left(\rho U^\top \vect{c}_i + A_i^\top \vect{x} - \vect{b}_i\right)\big)^\top & \lambda_i(\|\vect{d}\|_2^2 - r) -(\vect{a}_i^{(0)})^\top \vect{x} + b_i^{(0)} - \rho\vect{c}_i^\top \vect{y}_0
        \end{bmatrix} \succeq 0,
\end{align}
for $i=1,\ldots, m$, where $A_i:= \begin{bmatrix} \vect{a}_i^{(1)} & \cdots &  \vect{a}_i^{(k)} \end{bmatrix}$, and $\vect{b}_i = \begin{bmatrix} b_i^{(1)} & \cdots & b_i^{(k)} \end{bmatrix}^\top$, $i=1,\ldots,m$. Since the equivalence between \eqref{eqn:two-stage-cons} and \eqref{eqn:two-stage-cons-equiv} holds for arbitrary $\widehat{\vect{x}} = (\vect{x}, \vect{y}_0, U, \vect{\Theta}) \in \R^d \times \R^q \times \R^{q \times k} \times \prod_{p=1}^q \mathbb{S}^{k}$, we arrive at the convex SDP
\begin{align}\label{problem:ro-adjustable-sdp}
\min_{\substack{ \vect{x} \in \R^d,\,\vect{y_0} \in \R^q \\ U \in \R^{q \times k}, \, \Theta_p \in \mathbb{S}^{k}\\ \lambda_i \in \R}}
        & f(\vect{x}) + \sum_{i=1}^m\iota_{\R_+}(\lambda_i) + \iota_{\mathcal{C}}(\vect{x}) + \iota_{\mathcal{D}}(\vect{x}) 
        \\
        \st\,
        &
        \begin{bmatrix}
        \lambda_i I_k  - (1-\rho)\sum_{p=1}^q c_i^{(p)} \Theta_p & -\lambda_i \vect{d}  -\tfrac{1}{2} \left(\rho U^\top \vect{c}_i + A_i^\top \vect{x} - \vect{b}_i\right) \\[6pt]
        \big(-\lambda_i \vect{d}  -\tfrac{1}{2} \left(\rho U^\top \vect{c}_i + A_i^\top \vect{x} - \vect{b}_i\right)\big)^\top & \lambda_i(\|\vect{d}\|_2^2 - r) -(\vect{a}_i^{(0)})^\top \vect{x} + b_i^{(0)} - \rho\vect{c}_i^\top \vect{y}_0
        \end{bmatrix} \succeq 0,\notag
        \\ 
        & i=1,\ldots, m,\notag
\end{align}
where $\mathcal{D} : = \{g_j(\vect{x}) \le 0, \ j=1,\ldots,s\}$, which satisfies $\min \eqref{problem:ro-adjustable-qdr} = \min\eqref{problem:ro-adjustable-sdp}$.

The objective function of the convex SDP in \eqref{problem:ro-adjustable-sdp}
already decomposes into the terms
$F(\widetilde{\vect{x}})$, $E(\widetilde{\vect{x}})$, and $H(\widetilde{\vect{x}})$
defined in \Cref{eqn:FEHG-SDP2}. Moreover, by the definitions of $G$ and $K$
in \Cref{eqn:FEHG-SDP2,eqn:K-SDP2}, the LMI constraints in
\eqref{problem:ro-adjustable-sdp} can be written compactly as $ 
K\widetilde{\vect{x}} \in \prod_{i=1}^m \big(B_i + \mathbb{S}_+^{k+1}\big)$,
where $B_i$, $i=1,\dots,m$, are defined in \Cref{eqn:FEHG-SDP2}. Since $G$ is the indicator of the set $\prod_{i=1}^m (B_i + \mathbb{S}_+^{k+1})$,
these constraints are equivalent to requiring $G\big(K\widetilde{\vect{x}}\big) = 0$. Therefore, $\min\eqref{problem:ro-adjustable-qdr} = \min\eqref{problem:ro-adjustable-sdp} = \min_{\widetilde{\vect{x}}\in\widetilde{X}}
\big\{F(\widetilde{\vect{x}})
+ E(\widetilde{\vect{x}})
+ H(\widetilde{\vect{x}})
+ G\big(K\widetilde{\vect{x}}\big)\big\}$, as desired.
\end{proof}

\section{Computable SDP-Based Formulas for Proximal Operators} \label{sec:projections-via-SDP}

In this section, we derive formulas for computing projections onto closed convex sets, described in terms of SOS-convex polynomials, and proximal operators from the solution of associated SDPs.

Recall that $P_{\mathcal{D}}$ denotes the projection onto the closed convex set $\mathcal{D}$. We first derive a closed-form formula for $P_{\mathcal{D}}(\vect{v})$. For a given $\vect{v}\in \mathbb{R}^d$, we note that $P_{\mathcal{D}}(\vect{v})$ is the unique solution to the convex minimisation problem with SOS-convex constraints:
\begin{align*}\label{problem:prox-deltaD}\tag{PD}
    \overline{\mu}_{\vect{v}}= \inf_{\vect{x} \in \R^d} \left\{\|\vect{v} - \vect{x}\|_2^2 \, : \,  g_j(\vect{x}) \le 0,\, j=1,\ldots,s  \right\}.
\end{align*}

Let $h(\vect{x}):=\|\vect{v}-\vect{x}\|_2^2$, $\omega$ be the smallest even integer such that $\omega \ge \max_{j=1,\ldots,s}\deg g_j$, $\mathcal{\bf N}_\omega^d := \{(\alpha_1,\ldots,\alpha_d): \alpha_i \in \mathcal{N}_0, \, \sum_{i=1}^d \alpha_i\le \omega\}$ be a multi-index set where $\mathcal{N}_0$ is the set of nonnegative integers, and $s(d,\omega) := |\mathcal{\bf N}_\omega^d| = \binom{d+\omega}{\omega}$.

The Lagrangian dual of \eqref{problem:prox-deltaD} can be equivalently reformulated as 
\begin{equation}\label{DMP0}\tag{DMP$_{\vect{v}}$}
            \max_{\vect{\lambda}\in \R_+^s, \gamma \in \R, \ \sigma\in \Sigma_{\omega}^2} \bigg\{ \gamma : \; 
            \|\vect{v}- \cdot \|_2^2 + \sum_{j=1}^s \lambda_j g_j - \gamma =\sigma \bigg\}.
        \end{equation}
The duality between \eqref{problem:prox-deltaD} and \eqref{DMP0} is given below. To maintain a smooth flow in this section and minimise technical details related to polynomial optimisation, the proofs are deferred to the Appendix.

    \begin{proposition}[{\bf Duality for projection onto ${\mathcal{D}}$}]\label{B1}
        Let  $h(\vect{x})=\|\vect{v}-\vect{x}\|_2^2$ and $\mathcal{D} = \{\vect{x} \in \R^d : g_j(\vect{x}) \le 0, \, j=1,\ldots,s\}$, where $g_j$, $j=1,\ldots,s$, are SOS-convex polynomials. Assume that the Slater constraint qualification holds for \eqref{problem:prox-deltaD}. Let $\omega$ be an even integer such that $\omega \geq  \max_{j=1,\ldots,s} \deg g_j$. Then $\overline{\mu}_{\vect{v}}  = \max \eqref{DMP0}$.
    \end{proposition}

\begin{proof}
    The proof is given in Appendix~\ref{sec:sos}.
\end{proof}

As $h$ and $g_j, j=1,\ldots, s$, are all polynomials with degree at most $\omega$, one can write 
\[
h(\vect{x})=\sum_{\vect{\alpha} \in \mathcal{\bf N}_\omega^d} h_{\vect{\alpha}} \vect{x}^{\vect{\alpha}} \mbox{ and } g_{j}(\vect{x})= \sum_{\vect{\alpha} \in \mathcal{\bf N}_\omega^d} (g_{j})_{\vect{\alpha}} \vect{x}^{\vect{\alpha}},
\]
where $h_{\vect{\alpha}} \in \mathbb{R}$ and $(g_{j})_{\vect{\alpha}} \in \mathbb{R}$ are the real coefficients of $h$ and $g_j$ associated to the monomials $\vect{x}^{\vect{\alpha}}$, $\vect{\alpha} \in \mathcal{\bf N}_\omega^d$ respectively. The SDP reformulation of \eqref{DMP0} is given by
\begin{equation}\label{SDP0}\tag{SDP$_{\vect{v}}$}
\widehat{\mu}_{\vect{v}}:=\max_{\vect{\lambda} \in \R_+^s,\, \gamma \in \R,\,  Q \in \mathbb{S}^{\nu_0}_+}  \bigg\{ \gamma \, : \, h_{\vect{\alpha}} + \sum_{j=1}^s \lambda_j (g_j)_{\vect{\alpha}} - \gamma q_{\vect{\alpha}} = \tr(Q B_{
\vect{\alpha}}),
\ \vect{\alpha} \in { \mathcal{\bf N}_\omega^d}  \bigg\},
\end{equation}
where $q_{\vect{\alpha}} = 1$ for $\vect{\alpha} = \vect{0} \in { \mathcal{\bf N}_\omega^d}$ and $q_{\vect{\alpha}} = 0$ otherwise. Here, $B_{\vect{\alpha}} \in \mathbb{S}^{\nu_0}$ is the so-called moment matrices \cite{lasserre2009moments} with $\nu_0=s(d,\omega/2)$ (see Appendix~\ref{sec:sos} for details).
Its dual SDP is given by
\begin{equation}\label{MP}\tag{MP$_{\vect{v}}$}
\displaystyle\min_{\vect{y} \in \R^{s(d,\omega)}} \bigg\{\sum_{\vect{\alpha} \in \mathcal{\bf N}_\omega^d} h_{\vect{\alpha}} \vect{y}_{\vect{\alpha}} \, : \,   \sum_{\alpha \in \mathcal{\bf N}_\omega^d} (g_j)_{\vect{\alpha}} \vect{y}_{\vect{\alpha}} \leq 0, \, j=1,\ldots,s, \, \sum_{\alpha \in \mathcal{\bf N}_\omega^d} \vect{y}_{\vect{\alpha}} B_{\vect{\alpha}} \succeq 0,\, \vect{y}_{\vect{0}}=1 \bigg\}. 
\end{equation}

 Below, we give a computable SDP-based formula for the projection mapping to the convex set $\mathcal{D}=\{\vect{x}\in \R^d \, : \, g_j(\vect{x}) \le 0, \, j=1,\ldots, s\}$, where $g_j$'s are SOS-convex polynomials.

\begin{theorem}[{\bf Computable formula for $P_{\mathcal{D}}(\vect{v})$}]\label{formula}
For a fixed $\vect{v}\in \mathbb{R}^d$, let $h(\vect{x})=\|\vect{x}-\vect{v}\|_2^2$.  Let $g_j, j=1,\ldots, s$, be SOS-convex polynomials and let $\mathcal{D}=\{\vect{x}\in \R^d  :  g_j(\vect{x}) \le 0,\, j=1,\ldots, s\}$. Assume that $\max\eqref{SDP0}=\min\eqref{MP}$. Suppose that $\vect{y}_{\vect{v}}^*\in \R^{s(d,\omega)}$ is a solution  to the semidefinite program \eqref{MP} with $\vect{y}_{\vect{v}}^*=\big({y}_{\vect{v},\vect{\alpha}}^*\big)_{\vect{\alpha} \in \mathcal{\bf N}_\omega^d}$. Then,
\[
P_{\mathcal{D}}(\vect{v})= \big(y^*_{\vect{v},\vect{e}_1^{[d]}},\ldots,  y^*_{\vect{v},\vect{e}_d^{[d]}}\big)
\]
where $\vect{e}_i^{[d]}$, $i=1,\ldots,d$, are the multi-indices in $\mathcal{\bf N}_\omega^d$ whose $i$\textsuperscript{th} component is one, and zero otherwise. 
\end{theorem}  

\begin{proof}
    The proof is given in Appendix~\ref{sec:sos}.
\end{proof}

Following the above approach, we derive an easily computable formula for the projection onto the closed convex set, 
$\mathcal{D}_{\textup{quad}}=\{\vect{x}\in \R^d \,:\, \vect{x}^\top A_j \vect{x} + \vect{b}_j^\top \vect{x} + r_j \le 0, \, j=1,\ldots, s\}$, where $A_j \succeq 0$, $ \vect{b}_j\in\R^d$, $r_j \in \R$, $j=1,\ldots,s$. That is, we consider the projection problem:
\begin{align*}\label{problem:qp}\tag{PQ}
    \inf_{\vect{x} \in \R^d} \left\{\|\vect{v} - \vect{x}\|_2^2 \, : \,  \vect{x}^\top A_j\vect{x}+\vect{b}_j^\top\vect{x}+r_j \le 0, \; j=1,\ldots,s  \right\}.
\end{align*}
The Lagrangian dual of the projection problem \eqref{problem:qp} is the SDP given by: 
\begin{align*}\label{dq1}\tag{PQ1}
\max_{\vect{\lambda} \in \mathbb{R}_+^d, \, t \in \mathbb{R}} \Bigg\{
 t  \, : \, 
\begin{bmatrix}
\sum_{j=1}^s\lambda_jA_j+I_d
& \frac{1}{2}(\sum_{j=1}^s\lambda_j \vect{b}_j -  2 \vect{v}) \\[4pt]
\frac{1}{2}(\sum_{j=1}^s\lambda_j \vect{b}_j - 2\vect{v})^\top
& \|\vect{v}\|_2^2 +\sum_{j=1}^s\lambda_j r_j - t
\end{bmatrix}
\succeq 0. \Bigg\}
\end{align*}
Its Lagrangian dual problem becomes the following SDP:
\begin{align*}\label{dq2}\tag{PQ2}
\min_{S \in \mathbb{S}^{d}, \, \vect{u} \in \mathbb{R}^d} \quad
& \tr(S) - 2 \vect{v}^\top \vect{u} + \|\vect{v}\|_2^2 \\
\text{s.t.} \quad
& \begin{bmatrix}
1 & \vect{u}^\top \\
\vect{u} & S
\end{bmatrix} \succeq 0, \quad  \tr(A_j S) + \vect{b}_j^\top \vect{u} + r_j \le 0,
\quad j = 1, \ldots, s.
\end{align*}

\begin{theorem}[{\bf Computable formula for $P_{\mathcal{D}_{\textup{quad}}}(\vect{v})$}]\label{formula2}
Let $\vect{v}\in \mathbb{R}^d$ and $h(\vect{x})=\|\vect{x}-\vect{v}\|_2^2$.  Let 
$A_j\succeq0$, $ \vect{b}_j \in\R^d$, $r_j\in \R$, $j=1,\ldots,s$, and let $\mathcal{D}_{\textup{quad}}=\{\vect{x}\in \R^d \, : \, \vect{x}^\top A_j \vect{x} + \vect{b}_j^\top \vect{x} + r_j \le 0, \, j=1,\ldots, s\}$. If $(\vect{u}^{\vect{v}}, S^{\vect{v}})\in\R^d \times \mathbb{S}^d$ is a solution to the semidefinite program \eqref{dq2}, then
\[
P_{\mathcal{D}_{\textup{quad}}}(\vect{v}) = \vect{u}^{\vect{v}}.
\]
\end{theorem}
\begin{proof}
Let $(\vect{u}^{\vect{v}}, S^{\vect{v}})\in\R^d \times \mathbb{S}^{d}$ is a solution to \eqref{dq2}. To see the feasibility of $\vect{u}^{\vect{v}}$ for \eqref{problem:qp}, note  from the LMI in \eqref{problem:qp}, that $S^{\vect{v}}-\vect{u}^{\vect{v}}(\vect{u}^{\vect{v}})^\top \succeq 0$. Since $A_j \succeq 0$, $j=1,\ldots,s$, it follows that $
(\vect{u}^{\vect{v}})^\top A_j \vect{u}^{\vect{v}}
= \tr\big(A_j\, \vect{u}^{\vect{v}} (\vect{u}^{\vect{v}})^\top\big)
\le \tr(A_j S^{\vect{v}})$, $j=1,\ldots,s$.
Using the linear constraints in \eqref{problem:qp}, we obtain $
(\vect{u}^{\vect{v}})^\top A_j \vect{u}^{\vect{v}} + \vect{b}_j^\top \vect{u}^{\vect{v}} + r_j
\le \tr(A_j S^{\vect{v}}) + \vect{b}_j^\top \vect{u}^{\vect{v}} + r_j \le  0$, $j=1,\ldots,s$, and so $\vect{u}^{\vect{v}} \in \mathcal{D}_{\textup{quad}}$.

To verify optimality of $\vect{u}^{\vect{v}}$ for \eqref{problem:qp}, 
let $\vect{x} \in \mathcal{D}_{\textup{quad}}$. Then, $(\vect{x}, \vect{x} \vect{x}^\top)$ is feasible in \eqref{problem:qp} since
\[
\begin{bmatrix} 1 & \vect{x}^\top \\\vect{x} & \vect{x}\vect{x}^\top \end{bmatrix} \succeq 0 \quad \text{and} \quad \tr(A_j\vect{x}\vect{x}^\top)+\vect{b}_j^\top \vect{x}+r_j
= \vect{x}^\top A_j \vect{x} + \vect{b}_j^\top \vect{x} + r_j \le 0.
\]
Since $(\vect{u}^{\vect{v}}, S^{\vect{v}})$ is a solution to \eqref{dq2}, it follows that 
    \[
        \|\vect{v}\|_2^2-2\vect{v}^\top\vect{y^v}+\tr(S^{\vect{v}})
        \le  \|\vect{v}\|_2^2-2\vect{v}^\top\vect{x}+\tr({\vect{x}\vect{x}^\top}) =
        \|\vect{v}-\vect{x}\|_2^2.
    \]Now, because $S^{\vect{v}}-\vect{u}^{\vect{v}}(\vect{u}^{\vect{v}})^\top \succeq 0$, we have $\tr(S^{\vect{v}}) \ge \|\vect{u}^{\vect{v}}\|_2^2$. Hence, 
\begin{equation*}
\|\vect{u}^{\vect{v}}-\vect{v}\|_2^2 = \|\vect{v}\|_2^2-2\vect{v}^\top\vect{u}^{\vect{v}}+ \|\vect{u}^{\vect{v}}\|_2^2\le   \|\vect{v}\|_2^2-2\vect{v}^\top\vect{u}^{\vect{v}}+\tr(S^{\vect{v}}) \le 
\|\vect{x}-\vect{v}\|_2^2,
\end{equation*}
for every $\vect{x} \in \mathcal{D}_{\textup{quad}}$. Therefore,  $\vect{u}^{\vect{v}}$ is a solution for \eqref{problem:qp} and  $P_{\mathcal{D}_{\textup{quad}}}(\vect{v}) = \vect{u}^{\vect{v}}$.
\end{proof}

\section{Primal-Dual Proximal Splitting with Interwoven SDP}\label{sec:PD}

In this section, we present the primal-dual proximal splitting (PDPS) algorithm for solving the reformulated convex composite unconstrained problems together with the proximal operator formulas used in the iterative process. Recall that \eqref{problem:ro-adjustable} admits the following unconstrained composite reformulation:
\[
\min_{\widetilde{\vect{x}}} 
F(\widetilde{\vect{x}}) + E(\widetilde{\vect{x}}) + H(\widetilde{\vect{x}}) + G(K\widetilde{\vect{x}}),
\]
where $\widetilde{X}$ and $\widetilde{Y}$ are finite-dimensional spaces. We first put the composite reformulation in the following lifted\footnote{The lifting introduced here is not unique, and alternative representations of the composite structure are possible.} canonical form, required for the PDPS method: 
\begin{equation}\label{eqn:lifted-form}
    \min_{\widetilde{\vect{x}} \in \widetilde{X}}
    \; F(\widetilde{\vect{x}}) + \breve{G}(\breve{K}(\widetilde{\vect{x}})),
\end{equation}
where  the lifted operator and functional are given by 
\[
    \breve{K}(\widetilde{\vect{x}}) := (K\widetilde{\vect{x}},\, \widetilde{\vect{x}},\, \widetilde{\vect{x}}),
    \qquad
    \breve{G}(\widetilde{\vect{y}},\, \widetilde{\vect{z}}_1,\, \widetilde{\vect{z}}_2)
    := G(\widetilde{\vect{y}}) + E(\widetilde{\vect{z}}_1) + H(\widetilde{\vect{z}}_2),
\]
and the augmented space $\breve{Y}$ is given by $\breve{Y} = \widetilde{Y} \times \widetilde{X} \times \widetilde{X}$. This lifting allows the inclusion of \emph{interwoven} SDP substeps within the proximal evaluations of $H^*$, where semidefinite constraints arise. The resulting algorithm is outlined in \Cref{alg:CP-RO-ARO}. 

The convergence properties of the PDPS algorithm are well established;  see, e.g., \cite[Theorem~1]{chambolle2011first} or \cite{banert2025chambolle}, for the weak and ergodic convergence results.

\begin{algorithm}[H]
    \caption{PDPS Algorithm for Robust Convex Optimisation Problems}
    \label{alg:CP-RO-ARO}
    \begin{description}
        \item[Step 1] (\emph{Initialisation}) Choose parameters $\tau, \sigma > 0$ such that $\tau \sigma \|\breve K\|_{\mathbb{L}(\widetilde{X};\breve{Y})}^2 < 4/3$ and $\theta \in [0,1]$.  
        Initialise the primal variable $\widetilde{\vect{x}}^0$, the dual variable $\breve{\vect{y}}^0 = (\widetilde{\vect{y}}^0, \widetilde{\vect{z}}_1^0, \widetilde{\vect{z}}_2^0)$, and the extrapolated point $\overline{\vect{x}}^0 = \widetilde{\vect{x}}^0$.
        
        \item[Step 2] (\emph{Dual update}) Update $\breve{\vect{y}} = (\widetilde{\vect{y}}, \widetilde{\vect{z}}_1, \widetilde{\vect{z}}_2)$:
        \begin{align*}
            \widetilde{\vect{y}}^{n+1} &= \prox_{\sigma G^*}\big(\widetilde{\vect{y}}^{n} + \sigma K \overline{\vect{x}}^{n}\big)\\
            \widetilde{\vect{z}}_1^{n+1} & = \prox_{\sigma E^*}(\widetilde{\vect{z}}_1^n + \sigma\overline{\vect{x}}^{n})\\
            \widetilde{\vect{z}}_2^{n+1} & = \prox_{\sigma H^*}(\widetilde{\vect{z}}_2^n + \sigma\overline{\vect{x}}^{n})  \quad (\textbf{via SDP}).
        \end{align*}
        
        \item[Step 3] (\emph{Primal update}) 
        Update $\widetilde{\vect{x}}$: $
            \widetilde{\vect{x}}^{n+1}  = \prox_{\tau F} (\widetilde{\vect{x}}^n - \tau (K^* \widetilde{\vect{y}}^{n+1} + \widetilde{\vect{z}}_1^{n+1} + \widetilde{\vect{z}}_2^{n+1}))$.
        
        \item[Step 4] (\emph{Extrapolation}) Set $
        \overline{\vect{x}}^{n+1} =  \widetilde{\vect{x}}^{n+1} + \theta(\widetilde{\vect{x}}^{n+1} -  \widetilde{\vect{x}}^{n})$.
        
        \item[Step 5] (\emph{Stopping criterion}) If $\max\left\{\frac{\|\widetilde{\vect{x}}^{n+1} - \widetilde{\vect{x}}^{n}\|_{\widetilde{X}}}{\| \widetilde{\vect{x}}^{n}\|_{\widetilde{X}}}, \, \frac{\|\breve{\vect{y}}^{n+1} -\breve{\vect{y}}^{n}\|_{\breve{Y}}}{\|\breve{\vect{y}}^{n}\|_{\breve{Y}}}\right\} \leq \varepsilon$, 
        terminate and return $(\widetilde{\vect{x}}^{n+1}, \breve{\vect{y}}^{n+1})$.
        Otherwise, increment $n \leftarrow n+1$ and go to Step 2.
    \end{description}
\end{algorithm}

\textbf{Adjustable RO Problems: Proximal-Operator Formulas for  PDPS Iterations}. Recall that for the unconstrained convex composite reformulation \eqref{eqn:comp_prob_adjustable_qdr} of the adjustable robust optimisation problem \eqref{problem:ro-adjustable} introduced in \Cref{sec:aro} under ball uncertainty set, the proper, convex, lower semicontinuous functionals are given by
\[
F(\widetilde{\vect{x}}):= f(\vect{x}) + \sum_{i=1}^m \iota_{\R_+}(\lambda_i), \ E(\widetilde{\vect{x}})
:= \iota_{\mathcal{C}}(\vect{x}), \ H(\widetilde{\vect{x}})
:= \iota_{\mathcal{D}}(\vect{x}), \  \text{ and } \ 
G(\widetilde{\vect{y}}):= \sum_{i=1}^m\iota_{B_i + \mathbb{S}_+^{k+1}}(\Psi_i)
\]
where $\widetilde{X}:= \R^d \times \R^q \times \R^{q\times k} \times \prod_{p=1}^q\mathbb{S}^{k} \times \R^{m}$, and $\widetilde{Y}:= \prod_{i=1}^m\mathbb{S}^{k+1}$, $\widetilde{\vect{x}}
= \big(\vect{x},\,\vect{y}_0,\,U, \vect{\Theta}, \boldsymbol{\lambda}\big) \in \widetilde{X}$ with $\vect{\Theta}=(\Theta_1,\ldots,\Theta_q)$, $\vect{\lambda}=(\lambda_1,\ldots,\lambda_m)$,  $\vect{\widetilde{y}} = (\Psi_1, \ldots, \Psi_m) \in \widetilde{Y}$, and
\[B_i = \begin{bmatrix} 0 & -\tfrac12\boldsymbol{b}_i\\ -\tfrac12\boldsymbol{b}_i^\top & -b_i^{(0)}\end{bmatrix}, \ \vect{b}_i = \begin{bmatrix} b_i^{(1)} & \cdots & b_i^{(k)} \end{bmatrix}^\top, \ i=1,\ldots, m.\]

Moreover, recall that $\mathcal{D} = \{\vect{x} \in \R^d \, : \, g_j(\vect{x}) \le 0, \ j =1,\ldots, s\}$ where $g_j$'s are SOS-convex.

\begin{proposition}[{\bf Closed-form formulas for proximal operators: Two-stage ARO}]\label{prop:SDP_prox_ARO} Let $\widetilde{\vect{x}}
= \big(\vect{x},\,\vect{y}_0,\,U, \vect{\Theta}, \boldsymbol{\lambda}\big) \in \widetilde{X}$ with $\vect{\Theta}=(\Theta_1,\ldots,\Theta_q)$, $\vect{\lambda}=(\lambda_1,\ldots,\lambda_m)$,  $\vect{\widetilde{y}} = (\Psi_1, \ldots, \Psi_m) \in \widetilde{Y}$ and $\tau ,\sigma>0$. Assume that $f$ and $\iota_{\mathcal{C}}$ are proximable. The expressions for the proximal operators of the functionals defined for the two-stage ARO in \eqref{eqn:comp_prob_adjustable_qdr} are as follows.
\begin{enumerate}[label=(\roman*)]
\item \label{prop:SDP_prox_F_ARO} $
    \prox_{\tau F}(\widetilde{\vect{x}}) = (\prox_{\tau f}(\vect{x}),\vect{y}_0,U,\vect{\Theta}, P_{\R_+^m}(\boldsymbol{\lambda}))$, where $P_{\R_+^m}(\boldsymbol{\lambda})
    = \big(\max\{\lambda_1,0\},\ldots,\max\{\lambda_m,0\}\big)$.
\item \label{prop:SDP_prox_E_ARO} 
    $\prox_{\sigma E^*}(\widetilde{\vect{x}}) = \big(\vect{x} - \sigma P_{\mathcal{C}}(\tfrac{\vect{x}}{\sigma}),\vect{0},0_{q \times k},\vect{0},\vect{0}\big)$.
\item \label{prop:SDP_prox_G_ARO} 
    $\prox_{\sigma G^*}(\widetilde{\vect{y}}) = \Big(\Psi_1 -\sigma B_1-\sigma P_{\mathbb{S}_+^{k+1}}(\frac{\Psi_1-B_1}{\sigma}),\dots,\Psi_m-\sigma B_m-\sigma P_{\mathbb{S}_+^{k+1}}(\frac{\Psi_m-B_m}{\sigma})\Big)$.
\item \label{prop:SDP_prox_H_ARO}
    $\prox_{\sigma H^*}(\widetilde{\vect{x}}) = \big(\vect{x}-\sigma P_{\mathcal{D}}(\frac{\vect{x}}{\sigma}),\vect{0},0_{q \times k},\vect{0},\vect{0}\big)$, where $P_{\mathcal{D}}(\vect{v}) = (\vect{y}^*_{\vect{v},\vect{e}_1^{[d]}}, \ldots, \vect{y}^*_{\vect{v},\vect{e}_d^{[d]}})$ for any $\vect{v} \in \R^d$, where $\vect{y}_{\vect{v}}^*\in \R^{s(d,\omega)}$ is a solution  to the associated semidefinite program \eqref{MP} with $\vect{y}_{\vect{v}}^*=\big(\vect{y}_{\vect{v},\alpha}^*\big)_{\alpha \in \mathcal{\bf N}_\omega^d}$, $\omega$ is the smallest even integer such that $\omega \ge \max_{j=1,\ldots,s}\deg g_j$, and $\vect{e}_i^{[d]}$, $i=1,\ldots,d$, are the multi-indices in $\mathcal{\bf N}_\omega^d$ whose $i$\textsuperscript{th} component is one, and equals zero otherwise. 
\end{enumerate}
\end{proposition}
\begin{proof}
\textit{(i)} Given the separability of the functional $F$ in the variables $\vect{x},\,\vect{y}_0,U,\vect{\Theta}, \boldsymbol{\lambda}$, we use 
\cite[Proposition~24.11]{BauCom2017} to deduce that its proximal map acts componentwise, yielding
$\prox_{\tau F}(\widetilde{\vect{x}}) = (\prox_{\tau f}(\vect{x}),\vect{y}_0,U,\vect{\Theta}, P_{\R^m_+}(\boldsymbol{\lambda})).$
Moreover, the projection onto the nonnegative orthant $\R^m_+$ follows from \cite[Lemma~6.26]{beck2017first}.

\textit{(ii)}  Similarly, by a direct application of \cite[Proposition~24.11]{BauCom2017} and by using the fact that $\prox_{\iota_{\mathcal{C}}} = P_{\mathcal{C}}$, the proximal operator for $E$ is given by $\prox_{E}(\widetilde{\vect{x}}) = (P_{\mathcal{C}}(\vect{x}),\vect{y}_0,U,\vect{\Theta},\boldsymbol{\lambda}\big)$. By Moreau's decomposition (see, e.g., \cite[Proposition~24.8]{BauCom2017}), we then have
\[
\prox_{\sigma E^*}(\widetilde{\vect{x}}) = 
\vect{\widetilde{x}} - \big(\sigma \prox_{\sigma^{-1}\iota_\mathcal{C}}(\tfrac{\vect{x}}{\sigma}),\vect{y}_0,U,\vect{\Theta},\boldsymbol{\lambda}\big)
=\big(\vect{x} - \sigma P_{\mathcal{C}}(\tfrac{\vect{x}}{\sigma}),\vect{0},0_{q \times k},\vect{0},\vect{0}\big).
\]

\textit{(iii)} Firstly, notice that $G(\widetilde{\vect{y}}):= \sum_{i=1}^m\iota_{B_i + \mathbb{S}_+^{k+1}}(\Psi_i)$ is separable in the components $\Psi_i$, $i=1,\ldots,m$. By \cite[Proposition~24.11]{BauCom2017}, we have $
\prox_{G}(\widetilde{\vect{y}})
 = \big( \prox_{\iota_{B_i{+}\mathbb{S}_+^{k+1}}}(\Psi_i) \big)_{i=1}^m$. Thus, by Moreau's decomposition, the proximal operator for $\sigma G^*$ is given by
 \begin{align*}
 \prox_{\sigma G^*}(\widetilde{\vect{y}}) 
 & = (\Psi_1, \ldots, \Psi_m) - \sigma \Big(\prox_{\sigma^{-1}\iota_{B_1 + \mathbb{S}_+^{k+1}}}(\tfrac{\Psi_1}{\sigma}), \ldots, \prox_{\sigma^{-1}\iota_{B_m + \mathbb{S}_+^{k+1}}}(\tfrac{\Psi_m}{\sigma})\Big)\\
 &=
 \Big(\Psi_1 -\sigma B_1-\sigma P_{\mathbb{S}_+^{k+1}}(\tfrac{\Psi_1-B_1}{\sigma}),\dots,\Psi_m- \sigma B_m-\sigma P_{\mathbb{S}_+^{k+1}}(\tfrac{\Psi_m-B_m}{\sigma})\Big),
 \end{align*}
 where the last equality uses \cite[Proposition~29.1]{BauCom2017} to write $\prox_{\iota_{B_i + \mathbb{S}_+^{k+1}}}(\Psi_i) = B_i + P_{\mathbb{S}_+^{k+1}}(\Psi_i - B_i).$

\textit{(iv)} By \cite[Proposition~24.11]{BauCom2017}, we have $\prox_{H}(\widetilde{\vect{x}}) = (P_{\mathcal{D}}(\vect{x}),\vect{y}_0,U,\vect{\Theta},\boldsymbol{\lambda}\big)$, where the closed-form formula for $P_{\mathcal{D}}(\vect{v})$ for any $\vect{v} \in \R^d$ follows from \Cref{formula}. The proximal operator for $\sigma H^*$ then follows from Moreau's decomposition.
\end{proof}

Additionally, recall that the linear mapping $K:\widetilde{X}\to\widetilde{Y}\) is defined by $K\widetilde{\vect{x}}
:=\big( \Psi_1(\widetilde{\vect{x}}),\ldots,\Psi_m(\widetilde{\vect{x}})\big)$, where, for each $i=1,\ldots,m$,
\[
\Psi_i(\widetilde{\vect{x}})
=
\begin{bmatrix}
\lambda_i I_k - (1-\rho)\displaystyle\sum_{p=1}^q c_i^{(p)} \Theta_p
&
-\lambda_i \vect{d} - \tfrac{1}{2}\!\left(\rho U^\top \vect{c}_i + A_i^\top \vect{x}\right)
\\[4pt]
\big(-\lambda_i \vect{d} - \tfrac{1}{2}(\rho U^\top \vect{c}_i + A_i^\top \vect{x})\big)^\top
&
\lambda_i( \|\vect{d}\|_2^2 - r) - (\vect{a}_i^{(0)})^\top \vect{x} -\rho \vect{c}_i^\top \vect{y}_0
\end{bmatrix}
\in \mathbb{S}^{k+1},
\]
with $A_i:= \begin{bmatrix} \vect{a}_i^{(1)} & \cdots &  \vect{a}_i^{(k)} \end{bmatrix}$, $i=1,\ldots,m$. To define the adjoint of the linear mapping $K$, let $\vect{\widetilde{y}} = (\Psi_1,\dots, \Psi_m) \in \widetilde{Y}$ where
\begin{align*}
\Psi_i=
\begin{bmatrix}
\Psi_{i,11} & \Psi_{i,12} \\ (\Psi_{i,12})^\top & \Psi_{i,22}
\end{bmatrix}
\in\mathbb{S}^{k+1},
\end{align*}
$\Psi_{i,11}\in\mathbb{S}^k$, $\Psi_{i,12}
\in\R^k$, $\Psi_{i,22}\in\R$. Then the adjoint $K^*:\widetilde{Y}\to\widetilde{X}$ of the linear mapping $K$ in \Cref{prop:ro-adjustable-sdp} is given as $
K^*\vect{\widetilde{y}} = (\vect{x},\vect{y_0},U,\Theta,\vect{\lambda})$ where $\vect{x}=-\sum_{i=1}^m \big(A_i\Psi_{i,12} + \vect{a}_i^{(0)}\Psi_{i,22}\big)$, $\vect{y_0}=-\rho \sum_{i=1}^m \vect{c}_i\Psi_{i,22}$, $U=-\rho  \sum_{i=1}^m \vect{c}_i\, (\Psi_{i,12})^\top $, 
$\Theta=(\Theta_1,\ldots,\Theta_q)$ with $\Theta_p= -(1-\rho)\sum_{i=1}^m c_i^{(p)}\Psi_{i,11}$ for all $p=1,\ldots,q$, and $\vect{\lambda}=(\lambda_1,\ldots,\lambda_m)$ with $\lambda_i= \tr(\Psi_{i,11}) -2\vect{d}^\top \Psi_{i,12}+(\|\vect{d}\|_2^2-r)\Psi_{i,22}$ for all $i=1,\ldots,m$.

\section{Lot-Sizing with SOS-Convex Costs under Demand Uncertainty}\label{sec:numerics}

In this section, we study a version of the lot-sizing model that accounts for demand uncertainty and convex storage costs, examined in  \cite{boyd2017multi,woolnough2021exact,zhen2018adjustable}.

The setting involves a network of $N$ stores, where inventory allocations must be planned in advance to meet \textit{uncertain} customer demand. Stock may be delivered to each store at the start of the day for storage, or later transported among stores to satisfy realised demand. This formulation aligns with the adjustable robust optimisation framework, since part of the decision (the initial stock allocation) must be made \emph{before} the uncertain demand is revealed, while subsequent recourse actions (the inter-store transfers) can be adjusted \emph{after} the demand realisation.

To formulate an optimisation model for the lot-sizing problem, we let $x_i$, $i=1,\ldots, N$, denote the initial quantity of stock delivered to store $i$, incurring a storage cost parameterised by coefficients $\nu_i$, $\varphi_i$, and $\xi_i$, $i=1,\ldots,N$. Each store can hold up to $\Gamma$ units of stock at any time. Let $y_{ij}$, $i,j=1,\ldots,N$, represent the quantity of stock transported from store $i$ to store $j$, with a transportation cost $t_{ij}$ per unit, where we assume $t_{ii} = 0$.
The demand at store $i$, denoted by $w_i$, is uncertain at the start of the day and is known only to belong to an uncertainty set $\mathcal{B}$.

\textbf{Two-stage ARO Model}. To capture this demand uncertainty, we formulate the problem as a two-stage ARO model. The first-stage decisions correspond to the initial stock deliveries $\vect{x}=(x_1,\ldots, x_N)$, made before demand is realised. After the actual demand vector $\vect{w} = (w_1,\ldots, w_N)$ becomes known, the transportation quantities $y_{ij}(\vect{w})$, for $i,j = 1,\dots,N$, are determined adaptively to satisfy the demands at all stores.

The objective is to minimise the total cost of initial deliveries, storage, and transportation under the worst-case demand realisation. For some storage cost parameters $\nu_i,\, \varphi_i,\, \xi_i \ge 0$, $i=1,\ldots,N$, and transportation costs $t_{ij}$, $i=1,\ldots,N$, $j=1,\ldots,N$, the resulting two-stage ARO formulation is given by:
{\small \begin{align*}
\min_{\substack{\vect{x} \in \mathbb{R}^N, 
\\ y_{ij}(\cdot), \; i,j=1,\dots,N}}
& \quad \sum_{i=1}^{N}  (\nu_i x_i^4 + \varphi_i x_i^2 + \xi_ix_i) + \max_{\vect{w} \in \mathcal{B}} \bigg\{\sum_{i=1}^{N}\sum_{j=1}^N t_{ij} \ y_{ij}(\vect{w}) \bigg\}\\
\text{s.t.} \quad 
& x_i + \sum_{j=1}^{N} y_{ji}(\vect{w}) - \sum_{j=1}^{N} y_{ij}(\vect{w}) \;\geq\; w_i, 
&& \forall \vect{w}\in \mathcal{B}, \ i=1,\dots,N, \\
& y_{ij}(\vect{w}) \ge 0, \   
&& \forall \vect{w}\in \mathcal{B}, \ i,j=1,\dots,N, \\
& 0 \le x_i \le \Gamma,  && \ i=1,\ldots,N,
\end{align*}}
which is further equivalent to:
{\small \begin{align}\label{eqn:lot-sizing-ARO-equiv}\tag{LS}
\min_{\substack{\vect{x} \in \mathbb{R}^N, \, \varrho \in \R, \zeta\in \R, \\ y_{ij}(\cdot), \; i,j=1,\dots,N}}
& \quad  \varrho + \zeta\\
\text{s.t.} \quad 
& x_i + \sum_{j=1}^{N} y_{ji}(\vect{w}) - \sum_{j=1}^{N} y_{ij}(\vect{w}) \ge w_i, 
&& \forall \vect{w}\in \mathcal{B}, \ i=1,\dots,N, \notag\\
& \sum_{i=1}^{N}\sum_{j=1}^N t_{ij} \ y_{ij}(\vect{w}) \le \zeta, 
&& \forall \vect{w}\in \mathcal{B}, \notag\\
& y_{ij}(\vect{w}) \ge 0, \  
&& \forall \vect{w}\in \mathcal{B}, \ i,j=1,\dots,N,\notag\\
& 0 \le x_i \le \Gamma,  && \ i=1,\ldots,N,\notag\\
& \sum_{i=1}^{N}  (\nu_i x_i^4 + \varphi_i x_i^2 + \xi_ix_i)  - \varrho \le 0,\notag
\end{align}}
where $x_i$, $i=1,\ldots,N$, are the ``here-and-now" decision variables, $y_{ij}$, $i,j=1,\ldots,N$, are the ``wait-and-see" variables, $\varrho$ is an auxiliary epigraph  (``here-and-now") variable, $\zeta$ is a dummy (``here-and-now") variable to provide an uncertainty-free objective function, with uncertainty set $\mathcal{B} = \{\vect{w}\in\R^{N} : \norm{\vect{w} - \vect{d}}_2^2 \le r\}$, for some radius $r >0$ and nominal demand vector $\vect{d} \in \R^N$.
Note that the quartic term appearing in the constraints of \eqref{eqn:lot-sizing-ARO-equiv} represents a higher-order polynomial storage cost. This setup enables us to examine the performance of our proposed method when the feasible set is defined by SOS-convex inequalities.

The lot-sizing problem in \eqref{eqn:lot-sizing-ARO-equiv} can be expressed as \eqref{problem:ro-adjustable}, examined in Section~\ref{sec:aro}, as follows: 
\begin{align*}
    \min_{\vect{x}, \vect{y}(\cdot)}
        & 
        \ f(\vect{x}) \\
        \text{s.t.} \\
        & (\vect{a}_i^{(0)})^\top \vect{x} + \sum_{\ell=1}^k w^{(\ell)}(\vect{a}_i^{(\ell)})^\top \vect{x} + \vect{c}_i^\top \vect{y}(\vect{w})  \le b_i^{(0)} + \sum_{\ell=1}^{k} w^{(\ell)} b_i^{(\ell)},
        && \forall \vect{w} \in \mathcal{B}, \ i = 1,\ldots,N^2{+}N{+}1, \notag
        \\
        &
        \vect{x}\in \mathcal{C}, \ g_1(\vect{x}) \le 0, \notag
\end{align*}
where the variables, coefficients, and the objective function are identified as follows: the vector  $\vect{x}$ is identified with $(\vect{x}, \varrho, \zeta ) \in \R^{N} \times \R \times \R$, $f(\vect{x}) := \varrho + \zeta$, $$\vect{y}(\vect{w}):=(y_{11}(\vect{w}), \ldots , y_{N1}(\vect{w}) , y_{12}(\vect{w}) , \ldots , y_{NN}(\vect{w})) \in\R^{N^2},$$ $\vect{c}_i \in \R^{N^2}$ for $i=1,\ldots,N^2{+}N{+}1$, with entries
$\vect{c}_i=\vect{1}_N\otimes\vect{e}_i^{[N]} - \vect{e}_i^{[N]}\otimes\vect{1}_N$ for $i=1,\dots,N$ where $\vect{1}_N \in \R^N$ is the vector whose components are all equal to one, $\otimes$ denotes the Kronecker product, $\vect{c}_{N+1}=(t_{11},\ldots,t_{N1},  t_{12}, \ldots,  t_{NN})$, $\vect{c}_i=-\vect{e}_{i-(N+1)}^{[N^2]}$ for $i=N{+}2,\dots,N^2{+}N{+}1$, $\vect{a}_i^{(0)}\in\R^{N+2}$ for $i=1,\dots,N^2{+}N{+}1$ with $\vect{a}_i^{(0)}=(-\vect{e}_i^{[N]},0,0)\in \mathbb{R}^{N} \times \R \times \R$ for $i=1,\dots,N$, $\vect{a}_{N+1}^{(0)}=(\vect{0},0,-1) \in \mathbb{R}^{N} \times \R \times \R$, $\vect{a}_{i}^{(0)}=\vect{0} \in \R^{N+2}$ for $i=N{+}2,\ldots, N^2{+}N{+}1$, $\vect{a}_i^{(\ell)}=\vect{0}\in\R^{N+2}$ for $i=1,\dots,N^2{+}N{+}1$, $\ell=1,\dots,k$, $b_i^{(0)} = 0$, $i=1,\ldots,N^2{+}N{+}1$, $b_i^{(\ell)} = -\delta_{i\ell}$ for $i=1,\ldots,N$, $\ell = 1,\ldots,k$, where $\delta_{i\ell}$ denotes the Kronecker delta\footnote{The Kronecker $\delta_{i\ell}$ is equal to $1$ when $i=\ell$ and $0$ otherwise.}, $b_{i}^{(\ell)} = 0$ for $i = N{+}1, \ldots, N^2{+}N{+}1$, $\ell=1,\ldots,k$,
$\mathcal{C}=\{(\vect{x},\varrho,\zeta) \in \R^{N+2}: 0 \le \vect{x} \le \Gamma \}$, and $g_1(\vect{x},\varrho,\zeta):=\sum_{i=1}^{N}  (\nu_i x_i^4 + \varphi_i x_i^2 + \xi_ix_i)  - \varrho$.

Furthermore, by applying quadratic decision rule (QDR) and \Cref{prop:ro-adjustable-sdp}, we can express this problem in the convex composite form 
\begin{equation*}
    \min_{\widetilde{\vect{x}}\in\widetilde{X}} F(\widetilde{\vect{x}}) + E(\widetilde{\vect{x}}) + H(\widetilde{\vect{x}})+G(K\widetilde{\vect{x}})
\end{equation*}
where $\widetilde{X}=\R^{N} \times \R \times \R \times\R^{N^2}\times\R^{N^2\times N}\times \prod_{p=1}^{N^2}\mathbb{S}^N\times\R^{N+1+N^2}$, $\widetilde{\vect{x}} = \big(\vect{x},\varrho,\zeta,\vect{y}_0,U,\vect{\Theta}, \boldsymbol{\lambda}\big)$, $F(\widetilde{\vect{x}}) = \zeta + \varrho + \sum_{i=1}^{N^2{+}N{+}1} \iota_{\R^+}(\lambda_i)$, $E(\widetilde{\vect{x}}) = \sum_{i=1}^N \iota_{[0,\Gamma]}(x_i)$, $H(\widetilde{\vect{x}}) = \iota_{\mathcal{D}}(\vect{x}, \varrho,\zeta)$ where $\mathcal{D} = \{(\vect{x},\varrho,\zeta) \in \R^{N+2} \, : \, g_1(\vect{x},\varrho,\zeta) \le 0\}$, $G(\widetilde{\vect{y}}) = \sum_{i=1}^{N^2{+}N{+}1}\iota_{B_i + \mathbb{S}_+^{k+1}}(\Psi_i)$, and
$K(\widetilde{\vect{x}}) = (\Psi_1(\widetilde{\vect{x}}),\ldots,\Psi_{N^2{+}N{+}1}(\widetilde{\vect{x}}))$ where $B_i$ and $\Psi_i$ are as defined in \eqref{eqn:K-SDP2}, for all $i=1,\ldots, N^2{+}N{+}1$.

\subsection{Numerical Experiments:  Experimental design and performance metrics} To validate the tractability and practical performance of the proposed adjustable lot-sizing formulation and solution approach, we conduct a series of numerical experiments. All instances are solved using the PDPS method outlined in \Cref{alg:CP-RO-ARO} applied to the unconstrained convex reformulation of the lot-sizing problem.

\textit{Algorithmic setup}. In each configuration, the PDPS algorithm is executed until the stopping tolerance $\varepsilon = 10^{-5}$ is reached for the criterion described in \Cref{alg:CP-RO-ARO}. The step sizes $\tau$ and $\sigma$ are chosen to satisfy the standard convergence condition $\sigma \tau \|\breve K\|^2 < \frac{4}{3}$ for PDPS algorithm \cite{chambolle2011first,banert2025chambolle}, where $\|\breve K\|$ is the operator norm of the linear map $\breve K$ in the  composite model with the lifted canonical form (\ref{eqn:lifted-form}). To obtain a computationally efficient estimate, we employ 12 iterations of the power method \cite{golub2013matrix}. Then, we fix $\sigma = 0.5$ and set $\tau = \tfrac{1.3}{\sigma\|\breve K\|_{\text{est}}^2}$, where $\|\breve K\|_{\text{est}}$ denotes the estimated operator norm.

\textit{Performance metrics}. We report a combination of the resulting \textit{(i)} worst-case objective values, \textit{(ii)} CPU time and iteration counts, and \textit{(iii)} \emph{price of robustness} (PoR), defined as the percentage deviation of the worst-case solution from the nominal benchmark in terms of the minimised costs. 

\textit{Nominal Model}. More precisely, the nominal benchmark is computed by solving the nominal lot-sizing model given by:
\begin{align*}\label{nominal}\tag{NLS}
\min_{\vect{x},\, y_{ij}}\; & \sum_{i=1}^N (\nu_ix_i^4 + \varphi_ix_i^2 + \xi_ix_i) + \sum_{i=1}^N \sum_{j=1}^N t_{ij}\, y_{ij} \\
\text{s.t.}\quad 
& x_i + \sum_{j=1}^N y_{ji} - \sum_{j=1}^N y_{ij} \ge d_i, && i=1,\dots,N,\\
& y_{ij} \ge 0, && i,j=1,\dots,N,\\
& 0 \le x_i \le \Gamma, && i=1,\dots,N,
\end{align*}

where $\vect{d} = (d_1,\ldots, d_N)$ is the nominal demand vector. The PoR (in percentage) is then computed as $100\big(\frac{\chi_r - \chi_n}{\chi_n} \big)\%$,
where $\chi_r$ denotes the robust (worst-case) objective value, and $\chi_n$ is the nominal optimal value. A smaller PoR indicates a smaller degradation in cost due to robustness, and hence, reflects a more efficient trade-off between robust and nominal performance. In other words, the PoR quantifies the relative increase in objective (cost) required to immunise the solution against uncertainty. 

\textit{Experimental design.} Each simulation corresponds to a fixed network size $N$ and storage cost structure. Unless otherwise stated, the uncertainty set radius is set to $r=1$, the capacity bound is $\Gamma=1000$, the nominal demand vector is $\vect{d} = \vect{1}_N$, and $\rho=0.5$ in the quadratic decision rule. For experiments involving repeated runs, the cost parameters $\nu_i$, $\varphi_i$, $\xi_i$, $t_{ij}$, $i,j=1,\ldots,N$, were randomly generated from a uniform distribution and we reported averaged results over $100$ independent runs. For experiments where the network size $N$ is varied, the problem is solved only once, the cost parameters are instead determined and fixed according to a predefined rule that scales with $N$.

\textit{Computational environment and parallel implementation}. All simulations were executed using \textsc{MATLAB} R2024b on a mid-2024 MacBook Air with an M3 chip (8-core CPU, 10-core GPU) and 16GB RAM. For experiments involving repeated independent runs, the PDPS solver was implemented using \texttt{parfor} loops to allow for parallel computing.

\subsection{Experiment I: Numerical stability and price of robustness under linear costs}
 
In the first experiment, we investigate how the size of the ball uncertainty set affects the performance of the proposed PDPS method for the lot-sizing problem with (randomised) linear storage costs. More specifically, we consider the storage cost $\sum_{i=1}^N \xi_ix_i$ where $\xi_i$, $i=1,\ldots,N$, are randomised, along with the transaction costs parameters $t_{ij}$, $i,j=1,\ldots,N$. It is worth noting that under this linear cost structure, the dual update step of \Cref{alg:CP-RO-ARO} admits a closed-form formula that does not involve solving an SDP. The uncertainty set radius $r$ is varied over the values $r=0.0, 0.1, \ldots, 1.0$.

\Cref{fig:lot_sizing_summary} demonstrates that the proposed PDPS method solves the lot-sizing problem with linear storage costs across all uncertainty levels tested. As seen in \Cref{fig:lot_norm}, the algorithm reaches the prescribed stopping tolerance within a reasonable number of iterations, confirming tractability and numerical stability. \Cref{fig:lot_relchange} shows that the PoR increases monotonically with $r$.

\begin{figure}[H]
    \centering
    \begin{subfigure}[t]{0.4\textwidth}
        \centering
        \includegraphics[width=\linewidth]{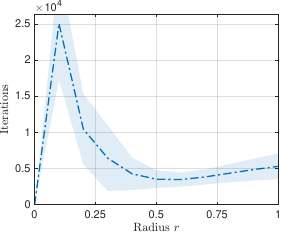}
        \caption{Number of PDPS iterations required.}
        \label{fig:lot_norm}
    \end{subfigure}
    \begin{subfigure}[t]{0.4\textwidth}
        \centering
        \includegraphics[width=\linewidth]{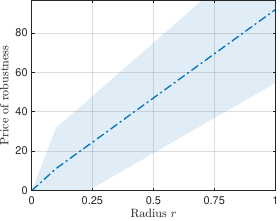}
        \caption{Price of robustness.}
        \label{fig:lot_relchange}
    \end{subfigure}
   \caption{Simulation results for the lot-sizing problem with linear transaction costs. Each panel shows the mean value across 100 simulations with one standard deviation as shaded bands, plotted as a function of the ball uncertainty radius.} 
    \label{fig:lot_sizing_summary}
\end{figure}

\subsection{Experiment II: Scalability of PDPS under linear costs} In this experiment, we investigate the scalability of the proposed PDPS method when solving the lot-sizing problem with linear storage costs. To assess how computational performance and solution quality evolve with problem dimension, we vary the network size, $N$, from 2 to 20, while fixing the uncertainty set size at $r=1.0$, the linear storage cost parameters (with $\xi_i = 1.0$, $i = 1, \ldots, N$), and the transaction cost parameters (with $t_{ij} = 2$ if $i \ne j$, and $t_{ij} = 0$ otherwise, for $i,j=1,\ldots,N$). 

\begin{table}[H]
    \centering
    \resizebox{0.8\textwidth}{!}{%
    \begin{tabular}{@{}ccccccc@{}}
        \toprule
        $N$ & $\#$ of primal-dual variables & Obj Val & PoR \% &  CPUT (s) & Iterations & Time/Iteration (s)\\
        \midrule
        2  & 147 & $3.9927$ & $99.6346$ & $0.4797$ & $2,828$ & $1.70\times 10^{-4}$ \\
        4  & 1,116 & $7.1230$ & $78.0744$ & $3.8116$ & $5,297$ & $7.19\times 10^{-4}$\\
        6  & 4,381 & $9.8561$ & $64.2682$ & $30.6909$ & $11,374$ & $2.70\times 10^{-3}$\\
        8  & 12,174 & $12.3544$ & $54.4302$ & $146.8857$ & $19,739$ & $7.44\times 10^{-3}$ \\
        10 & 27,495 & $14.6174$ & $46.1744$ & $520.8207$ & $29,302$ & $1.77\times 10^{-2}$\\
        15 & 125,350 & $18.0563$ & $20.3753$ & $2,845.0710$ &  $28,371$ & $1.02\times 10^{-1}$\\
        20 & 375,780 & $23.0031$ & $15.0156$ & $19,800.2729$ & $54,555$ & $3.63\times 10^{-1}$\\ 
        \bottomrule
    \end{tabular}}
    \caption{Summary of results for the experiments with linear transaction cost.}
    \label{tab:linear_cost}
\end{table}

Table~\ref{tab:linear_cost} summarises the numerical results under linear storage costs with varying network size $N$. The reported objectives, PoR, CPU times, and iteration counts illustrate that the proposed PDPS method remains numerically tractable as the problem dimension increases, with computational effort growing moderately with $N$.

It is worth noting that when the storage and transportation costs are linear, the problem can be solved directly by formulating the robust lot-sizing problem as a single SDP, as demonstrated in \cite[Section~4]{woolnough2021exact}. However, this direct SDP approach becomes computationally prohibitive as the problem dimension increases, since it can only handle instances up to $N=8$ due to the growing dimension of the SDP. In contrast, the proposed PDPS frameworks offer a scalable alternative by decomposing the problem and enforcing the semidefinite constraints iteratively through simple projection steps. This splitting-based approach significantly reduces computational and memory requirements, making it possible to solve larger problem instances.

\subsection{Experiment III: Evaluation of PDPS-SDP approach under SOS-convex costs}
In this experiment, we investigate the numerical tractability of the proposed PDPS-SDP method when solving the lot-sizing problem with higher degree SOS-convex polynomial storage costs. More specifically, we consider the storage cost $\sum_{i=1}^N \nu_ix_i^4$ where $\nu_i = 1.0$, $i=1,\ldots,N$. Note that, under this quartic cost structure, the dual update step of \Cref{alg:CP-RO-ARO} involves solving an SDP, interwoven with the PDPS.  We also fixed the ball uncertainty set size at $r=1.0$, and the transaction cost parameters with $t_{ij} = 2$ if $i \ne j$, and $t_{ij} = 0$ otherwise, for $i,j=1,\ldots,N$.

\textbf{Numerical stability}. \Cref{fig:evolution} illustrates the convergence behaviour of the proposed PDPS method with interwoven SDP calculations for the lot-sizing problem involving quartic storage costs. The figure shows the evolution of the stopping criterion (in logarithmic scale) for $N=6$ and $N=8$. 

 \begin{figure}[H]
    \centering
    \begin{subfigure}[t]{0.4\textwidth}
        \centering
        \includegraphics[width=\linewidth]{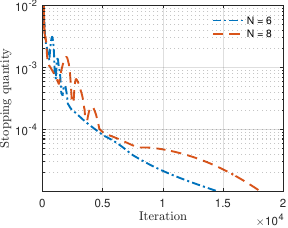}
    \end{subfigure}
    \caption{Evolution of the stopping quantity in \Cref{alg:CP-RO-ARO} for the lot-sizing problem solved via PDPS with interwoven SDP calculations for handling quartic storage costs, shown for $N=6$ and $N=8$.}
    \label{fig:evolution}
\end{figure}

\begin{table}[H]
\centering
\resizebox{0.85\textwidth}{!}{%
\begin{tabular}{cccccccc}
\hline
\multicolumn{1}{c}{\multirow{2}{*}{$N$}} & $\#$ of primal-dual & \multicolumn{2}{c}{Interwoven SDP} & \multirow{2}{*}{Obj Val} & \multirow{2}{*}{CPUT (s)} & \multirow{2}{*}{Iterations} & \multirow{2}{*}{Time/Iteration (s)} \\ \cline{3-4}
\multicolumn{1}{c}{} & variables & \multicolumn{1}{r}{Variables} & \multicolumn{1}{r}{Constraints} &  &  &  &  \\ \hline
2 & 147 & 21 & 14 & $16.9220$ & $495.2853$ & $9,808$ & $5.05\times 10^{-2}$\\
4 & 1,116 & 120 & 69 & $19.8737$ & $675.9231$ & $12,729$ & $5.31\times 10^{-2}$ \\
6 & 4,381 & 406 & 209 & $21.9356$ & $870.2896$ & $14,549$ & $5.98\times 10^{-2}$ \\
8 & 12,174 & 1,035 & 494 & $24.2206$ & $1,485.5357$ & $18,100$ & $8.20\times 10^{-2}$ \\ 
10 & 27,495 & 2,211 & 1,000 & $28.6495$ & $4,555.8002$ & $28,808$ & $1.58\times 10^{-1}$ \\ \bottomrule
\end{tabular}}
\caption{PDPS with interwoven SDP (quartic cost) for $r=1$.}
\label{tab:quartic_cost}
\end{table}

\textbf{Overall computational performance}.  To assess how computational performance evolves with problem dimension for the lot-sizing problem with nonlinear storage cost, we vary the network size, $N$, from 2 to 10, while fixing the uncertainty set size at $r=1.0$ under the quartic storage costs. 

 \Cref{tab:quartic_cost} highlights the computational behaviour of the proposed PDPS framework with interwoven SDP calculations under SOS-convex (quartic) storage costs. The results confirm that the PDPS approach with interwoven SDP-based calculations effectively handles nonlinear-SOS convex costs, maintaining tractability for moderate-scale robust lot-sizing problems.  
 
 However, for any fixed $N$, the quartic cost model requires substantially more iterations than the linear cost formulation. 
 Moreover, comparing the Time/Iteration columns of \Cref{tab:linear_cost,tab:quartic_cost}, each iteration for the quartic cost model is about one to two orders of magnitude more expensive (in terms of CPU time) than the linear cost model, revealing the additional computing burden introduced by the SDP-based proximal calculations.

\subsection{Strengths and limitations of the proposed computational framework}
A key strength of the proposed PDPS scheme, interwoven with SDP, lies in its ability to extend the computational tractability to a broad class of robust convex optimisation problems involving nonlinear SOS-convex polynomial objectives or constraints. This extension significantly broadens the scope of both the PDPS scheme and robust convex optimisation. By decomposing the original problem and enforcing computations through simple iterative projection steps, the framework offers a moderately scalable alternative to solving the entire robust formulation as a single large-scale SDP. The decomposition reduces computational and memory requirements, thereby enabling the solution of larger problem instances while preserving convergence guarantees and ensuring numerical stability.

Despite these advantages, a notable limitation of this approach is its continued reliance on SDPs within the proximal step when handling nonlinear SOS-convex objective or constraint functions. However, this reliance is mitigated by the fact that the SDP computation is confined to the proximal step, rather than distributed across the entire optimisation process. As a result, the heavy SDP workload is isolated within a well-structured subroutine, allowing the overall PDPS framework to remain computationally tractable and scalable. Nevertheless, for large-scale problems involving high-degree or high-dimensional SOS-convex polynomials, the size of these SDP subproblems may still pose scalability and computational efficiency challenges.

\section{Conclusion}\label{sec:conclusion}

In this work, we introduced a primal-dual splitting framework that integrates semidefinite programming techniques and tools from semi-algebraic geometry to address computational tractability of a broad class of robust convex optimisation problems with SOS-convex polynomial constraints. By reformulating robust constraints as LMIs, we transformed the original problems into unconstrained convex composite formulations sharing the same optimal values. Furthermore, we established an efficiently computable SDP-based formula for projections onto closed convex sets defined by SOS-convex inequalities, enabling the integration of SDPs within the first-order proximal splitting algorithms for solving robust optimisation problems. The proposed framework was validated through numerical experiments on lot-sizing problems with linear as well as SOS-convex storage and inventory costs under demand uncertainty, demonstrating its applicability and computational tractability.

\begin{appendices}

\section{SOS-Convex polynomials: Technical Details and Proofs}  \label{sec:sos}

We provide some basic details for the relationship between sum-of-squares polynomials and linear matrix inequalities \cite{lasserre2009moments}.

A monomial over $\vect{x} \in \R^d$ of degree $\omega$ is $\vect{x}^{\vect{\alpha}} = x_1^{\alpha_1} x_2^{\alpha_2} \ldots x_d^{\alpha_d}$ with $\omega = \sum_{i=1}^d \alpha_i$ and multi-index $\vect{\alpha} = (\alpha_1, \ldots, \alpha_d)$. The canonical basis is \begin{align*}
    \vect{y}(\vect{x}) := (1, x_1, \ldots, x_d, x_1^2, x_1x_2, \ldots, x_d^2, \ldots, x_1^{\omega}, \ldots, x_d^{\omega})^\top,
\end{align*}which is of dimension $s(d,\omega) := \binom{d+\omega}{\omega}$. 

Consider a polynomial $f$ on $\R^d$ with degree at most $2\omega$. It can be expressed as $f(\vect{x}) = \sum_{\vect{\alpha} \in \mathcal{\bf N}_{2\omega}} f_{\vect{\alpha}} x^{\vect{\alpha}}$ where $f_{\vect{\alpha}}$ is the $\vect{\alpha}$\textsuperscript{th} coefficient of $f$ and $\mathcal{\bf N}_{2\omega}^d = \{(\alpha_1, \ldots, \alpha_d) : \alpha_i \in \mathcal{N}_0, \, \sum_{i=1}^d \alpha_i \le 2\omega\}$ is a multi-index set, where $\mathcal{N}_0$ is the set of nonnegative integers. By \cite[Proposition 2.1]{lasserre2009moments}, $f$ is SOS if and only if there exists $Q \in \mathbb{S}^{s(d,\omega)}_+$ such that $f(\vect{x}) = \vect{y}(\vect{x})^\top Q \vect{y}(\vect{x})$.

Thus, if we write
\begin{align}\label{lmi}
    f(\vect{x}) = \vect{y}(\vect{x})^\top Q \vect{y}(\vect{x}) = \tr\big(Q \vect{y}(\vect{x}) \vect{y}(\vect{x})^\top\big) = \sum_{\vect{\alpha} \in \mathcal{\bf N}_{2\omega}^d} \tr(Q B_{\vect{\alpha}}) \vect{x}^{\vect{\alpha}}, \ \text{ for all } \ \vect{x} \in \R^d,
\end{align}
for some appropriate matrices $B_{\vect{\alpha}} \in \mathbb{S}^{s(d,\omega)}$, then by comparing the coefficients of each side of equation \eqref{lmi}, we obtain that $f$ is SOS if and only if there exists $Q \in \mathbb{S}^{s(d,\omega)}_+$ such that $f_{\vect{\alpha}} = \tr(QB_{\vect{\alpha}})$ for all $\vect{\alpha} \in \mathcal{\bf N}_{2\omega}^d$. Hence, checking whether $f$ is SOS amounts to finding $Q \in \mathbb{S}^{s(d,\omega)}_+$ such that $\tr(QB_{\vect{\alpha}})=f_{\vect{\alpha}}$, for $\vect{\alpha}\in\mathcal{\bf N}_{2\omega}^d$.

\begin{proposition}[\bf{\Cref{B1}}]\label{AB1}
    Let  $h(\vect{x})=\|\vect{v}-\vect{x}\|_2^2$ and $\mathcal{D} = \{\vect{x} \in \R^d : g_j(\vect{x}) \le 0, \, j=1,\ldots,s\}$, where $g_j$, $j=1,\ldots,s$, are SOS-convex polynomials. Let $\omega$ be an even integer such that $\omega \ge  \max_{j=1,\ldots,s} \deg g_j$. Assume that the Slater constraint qualification holds for \eqref{problem:prox-deltaD}.  Then,
    \begin{equation}\label{DMP}
        \inf{\eqref{problem:prox-deltaD}} = \max_{\vect{\lambda}\in \R_+^s, \gamma\in \R, \sigma\in \Sigma_{\omega}^2} \bigg\{ \gamma : \; 
        \|\vect{v}- \cdot \|_2^2 + \sum_{j=1}^s \lambda_j g_j - \gamma =\sigma  \bigg\}.
    \end{equation}
\end{proposition}
\begin{proof} 
Let $\overline{\mu}_{\vect{v}}: = \inf \eqref{problem:prox-deltaD}$. Then $\overline{\mu}$ is finite as $h$ is a coercive polynomial and the problem \eqref{problem:prox-deltaD} has a minimiser. By the Lagrangian duality theorem  under the Slater constraint qualification
\cite[Corollary 4.1]{jeyakumar2008constraint}, we have
{\small\begin{align*}
\overline{\mu}_{\vect{v}} 
&
= \inf \eqref{problem:prox-deltaD} 
= \inf_{\vect{x} \in \R^d} \Big\{  \|\vect{v} - \vect{x}\|_2^2 \, :  \,  g_j(\vect{x}) \le 0, \, j=1,\ldots,s \Big\}
=
\max_{\lambda_j \geq 0} \inf_{\vect{x} \in \mathbb{R}^d} \Big \{  \|\vect{v} - \vect{x}\|_2^2 + \sum_{j=1}^s \lambda_j g_j(\vect{x}) \Big\}.
\end{align*}}So, there exist $\lambda_j^\star \ge 0$, $j=1,\dots,s$, such that $
\|\vect{v} - \vect{x}\|_2^2 + \sum_{j=1}^s \lambda_j^\star g_j(\vect{x}) - \overline{\mu}_{\vect{v}} \ge 0$, for all $\vect{x} \in \R^d$. Since $\lambda_j^\star \geq 0$, $j=1,\ldots,s$, and $f$, $g_j$, $j=1,\ldots,s$, are SOS-convex polynomials, 
it follows that $\|\vect{v} - \cdot \|_2^2 + \sum_{j=1}^s \lambda_j^\star g_j - \overline{\mu}$
is a nonnegative SOS-convex polynomial. Consequently, by \Cref{prop:nn_sos_conv},
$\|\vect{v} - \cdot \|_2^2 + \sum_{j=1}^s \lambda_i^\star g_j- \overline{\mu}_{\vect{v}} \in \Sigma_{\omega}^2$. 
Thus, $(\vect{\lambda}^\star, \overline{\mu}_{\vect{v}})$ with $\vect{\lambda} = (\lambda_1^\star,\ldots, \lambda_s^\star)$, is a feasible point for \eqref{DMP}, and so $\overline{\mu}_{\vect{v}} \le \max \eqref{DMP}$. 

The weak duality holds by construction.  To see this, let $\vect{x}$ be feasible for \eqref{problem:prox-deltaD}. For any feasible point $(\vect{\lambda}, \gamma)$ of \eqref{DMP}, we have $\|\vect{v} - \cdot\|_2^2 + \sum_{j=1}^s \lambda_j g_j- \gamma \in \Sigma_{\omega}^2$. In other words, $
\|\vect{v} - \vect{x}\|_2^2- \gamma \ge  -\sum_{j=1}^s \lambda_j g_j(\vect{x})  \ge 0$,
where the last inequality follows from the feasibility of $\vect{x}$ for \eqref{problem:prox-deltaD}. Thus, $\|\vect{v} - \vect{x}\|_2^2 \ge \gamma$. Therefore, $\inf\eqref{problem:prox-deltaD} = \overline{\mu}_{\vect{v}} \ge \max \eqref{DMP}$, and so the conclusion holds.
\end{proof}

We make use of the following generalised Jensen's inequality \cite[Theorem~2.6]{lasserre2009convexity} to prove \Cref{formula}. To do this, given $\vect{y} =(\vect{y}_{\vect{\alpha}})_{\vect{\alpha} \in \mathcal{\bf N}_\omega^d}$, recall the Riesz mapping  $L_{\vect{y}} : \R[\vect{x}] \to \R$ is given by $L_{\vect{y}}(f) = \sum_{\vect{\alpha} \in \mathcal{\bf N}_{\omega}^d} f_{\vect{\alpha}} y_{\vect{\alpha}}$ for any $f \in  \R[\vect{x}]$ with degree at most $\omega$.

\begin{proposition}[{\bf Jensen's inequality for SOS-convex polynomials \cite[Theorem~2.6]{lasserre2009convexity}}]\label{prop:jensens}
    Let $f$ be an SOS-convex polynomial on $\R^d$ of an even degree $\omega$, $\vect{y} = (y_{\vect{\alpha}})_{\alpha \in \mathcal{\bf N}_{\omega}^d}$ satisfy $y_{\vect{0}} = 1$ and $\sum_{\alpha \in \mathcal{\bf N}_{\omega}^d} y_{\vect{\alpha}} B_{\vect{\alpha}} \succeq 0$, $B_{\alpha} \in \mathbb{S}^{s(d,\omega/2)}$, and $L_{\vect{y}} : \R[\vect{x}] \to \R$ be the linear functional $L_{\vect{y}}(h) = \sum_{\alpha \in \mathcal{\bf N}_\omega^d}f_{\vect{\alpha}} y_{\vect{\alpha}}$. Then, $L_{\vect{y}} (f) \ge f(L_{\vect{y}} (\mathbf{x}_1), \ldots, L_{\vect{y}} (\mathbf{x}_d))$, where $\mathbf{x}_i$ denotes the polynomial that maps a vector $\vect{x} \in \R^d$ to the $i$\textsuperscript{th} coordinate $x_i$, $i=1,\ldots,d$.
\end{proposition}

\begin{theorem}[{\bf\Cref{formula}}]\label{formula1}
For a fixed $\vect{v}\in \mathbb{R}^d$, let $h(\vect{x})=\|\vect{x}-\vect{v}\|_2^2$.  Let $g_j, j=1,\ldots, s$, be SOS-convex polynomials and let $\mathcal{D}=\{\vect{x}\in \R^d \,: \,g_j(\vect{x}) \le 0,\,  j=1,\ldots, s\}$. Assume that 
$\overline{\mu}_{\vect{v}} = \max\eqref{SDP0} = \min\eqref{MP}$. Suppose that $\vect{y}_{\vect{v}}^* \in \R^{s(d,\omega)}$ is a solution  to the semidefinite program \eqref{MP} with $\vect{y}_{\vect{v}}^*=\big(y_{\vect{v},\vect{\alpha}}^*\big)_{\vect{\alpha} \in \mathcal{\bf N}_{\omega}^d}$.
Then, $
P_{\mathcal{D}}(\vect{v})= (y^*_{\vect{v},\vect{e}_1^{[d]}}, \ldots, y^*_{\vect{v},\vect{e}_d^{[d]}})$,
where $\vect{e}_i^{[d]}$, $i=1,\ldots,d$, are the multi-index in $\mathcal{\bf N}_\omega^d$ whose $i$\textsuperscript{th} component is one, and zero otherwise.  
\end{theorem}
\begin{proof} Let $\vect{u}_i$ denote the polynomial that maps a vector $\vect{u} \in \R^d$ to its $i$\textsuperscript{th} coordinate $u_i$, $i=1,\ldots,d$, and $\overline{\vect{x}}^* =(L_{\vect{y}_{\vect{v}}^*}(\vect{u}_1), \ldots, L_{\vect{y}_{{\vect{v}}}^*}(\vect{u}_d))= (y^*_{\vect{v},\vect{e}_1^{[d]}}, \ldots, y^*_{\vect{v},\vect{e}_d^{[d]}}) \in \mathbb{R}^d$. We will show that $\overline{\vect{x}}^*$ is the solution to \eqref{problem:prox-deltaD}.

We note that the function $h$ is coercive and strongly convex, and so, problem \eqref{problem:prox-deltaD} has a unique minimiser and $\overline{\mu}_{\vect{v}}$ is finite. Let $\vect{y}_{\vect{v}}^* =\big(y_{\vect{v},\vect{\alpha}}^*\big)_{\vect{\alpha} \in \mathcal{\bf N}_\omega^d} \in \R^{s(d,\omega)}$ be a solution  to the semidefinite program \eqref{MP}. Then, 
\[
L_{\vect{y}_{{\vect{v}}}^*}(g_j) = \sum_{\vect{\alpha} \in \mathcal{\bf N}_\omega^d} (g_j)_{\vect{\alpha}} \, y_{\vect{v},\vect{\alpha}}^* \leq 0, \ j=1,\ldots,s, \ \text{ and } \ \sum_{\vect{\alpha} \in \mathcal{\bf N}_\omega^d} y_{\vect{v},\vect{\alpha}}^* B_{\vect{\alpha}} \succeq 0 \ \text{ with }\ y_{\vect{v},\vect{0}}^*=1. 
\]
Moreover, by the assumption, one has $\overline{\mu}_{\vect{v}}=\max\eqref{SDP0}=\min\eqref{MP}=L_{\vect{y}_{{\vect{v}}}^*}(h)$.
Now, applying Jensen's inequality for SOS-convex polynomials (see Proposition \ref{prop:jensens}) to $g_j$ together with the vector $\vect{y}_{\vect{v}}^*\in \R^{s(d,\omega)}$, one obtains that, for all $j=1,\ldots,s$, 
\begin{eqnarray*}
L_{\vect{y}_{{\vect{v}}}^*}(g_j)  \geq  g_j (L_{\vect{y}_{\vect{v}}^*}(\vect{ u}_1),\ldots,L_{\vect{y}_{{\vect{v}}}^*}(\vect{u}_d)) 
 =  g_j(y^*_{\vect{v},\vect{e}_1^{[d]}}, \ldots, y^*_{\vect{v},\vect{e}_d^{[d]}}) = g_j(\overline{\vect{x}}^*),
\end{eqnarray*}
where the first equality follows from the definitions of the Riesz mapping and the mapping $\vect{u}_i$,  $i=1,\ldots,d$, and the second equality is by the definition of $\overline{\vect{x}}^*$. This implies that, for all $j=1,\ldots, s$, 
         $0\geq L_{\vect{y}{^*}}(g_j)\geq g_j(\overline{\vect{x}}^*)$, 
and so, $\overline{\vect{x}}^*$ is feasible for the problem \eqref{problem:prox-deltaD}. Similarly, noting that $h$ is also an SOS-convex polynomial, one also has $
\overline{\mu}_{\vect{v}}= L_{\vect{y}_{{\vect{v}}}^*}(h) \ge h(\overline{\vect{x}}^*).$ This then implies that $
\overline{\vect{x}}^*=(L_{\vect{y}_{\vect{v}}^*}(\vect{u}_1),\ldots,L_{\vect{y}_{\vect{v}}^*}(\vect{u}_d))=(y^*_{\vect{v},\vect{e}_1^{[d]}}, \ldots, y^*_{\vect{v},\vect{e}_d^{[d]}})$ is a minimiser of \eqref{problem:prox-deltaD}.
\end{proof}

\end{appendices}
\vspace{-0.49cm}
\singlespacing
\end{document}